\documentclass[a4paper,10pt,reqno]{amsart}

\usepackage[utf8]{inputenc}
\usepackage[T1]{fontenc}
\usepackage{amsthm}
\usepackage{amsmath}
\usepackage{amssymb}
\usepackage[inline]{enumitem}
\usepackage{comment}
\usepackage{hyperref}
\usepackage{fancyhdr}
\usepackage{mathrsfs}
\usepackage{stmaryrd}
\usepackage[normalem]{ulem}
\usepackage{xcolor}

\theoremstyle{definition}

\newtheorem{theorem}{Theorem}[section]
\newtheorem{lemma}[theorem]{Lemma}
\newtheorem{proposition}[theorem]{Proposition}
\newtheorem{corollary}[theorem]{Corollary}

\theoremstyle{definition}
\newtheorem{definition}[theorem]{Definition}

\newtheorem{remark}[theorem]{Remark}

\newtheorem{remarks}[theorem]{Remarks}

\newtheorem{claim}{Claim}

\definecolor{blue-url}{RGB}{0,0,100}
\definecolor{red-url}{RGB}{100,0,0}
\definecolor{green-url}{RGB}{0,100,0}
\definecolor{light-yellow}{RGB}{255,255,128}
\definecolor{light-blue}{RGB}{193,255,255}
\definecolor{light-red}{RGB}{239,83,80}

\hypersetup{
	pdftitle={Idempotent factorizations},
	pdfauthor={Cossu and Tringali},
	pdfmenubar=false,
	pdffitwindow=true,
	pdfstartview=FitH,
	colorlinks=true,
	linkcolor=blue-url,
	citecolor=green-url,
	urlcolor=red-url
}

\renewcommand{\emptyset}{\varnothing}
\renewcommand{\setminus}{\smallsetminus}
\renewcommand{\,}{\kern 0.1em}

\DeclareMathOperator{\rann}{\textup{r.\textsc{a}nn}}
\DeclareMathOperator{\lann}{\textup{l.\textsc{a}nn}}
\DeclareMathOperator{\fix}{\textup{\textsc{f}ix}}
\DeclareMathOperator{\rfix}{\textup{r.\textsc{f}ix}}
\DeclareMathOperator{\im}{\textsc{I}m}
\DeclareMathOperator{\End}{\textsc{e}nd}
\DeclareMathOperator{\rk}{rk}
\DeclareMathOperator{\hgt}{ht}
\DeclareMathOperator{\udim}{\textup{udim}}

\DeclareMathOperator{\rdim}{\textup{r.udim}}

\providecommand\llb{\llbracket}
\providecommand\rrb{\rrbracket}

\providecommand\id{{\rm id}}

\newcommand{\evid}[1]{\textsf{#1}}

{\newline\vspace{\abovedisplayskip}\hbox to \textwidth\bgroup\hss$\displaystyle}
{$\hss\egroup\vspace{\belowdisplayskip}}

\makeatletter
\DeclareFontFamily{OMX}{MnSymbolE}{}
\DeclareSymbolFont{MnLargeSymbols}{OMX}{MnSymbolE}{m}{n}
\SetSymbolFont{MnLargeSymbols}{bold}{OMX}{MnSymbolE}{b}{n}
\DeclareFontShape{OMX}{MnSymbolE}{m}{n}{
	<-6>  MnSymbolE5
	<6-7>  MnSymbolE6
	<7-8>  MnSymbolE7
	<8-9>  MnSymbolE8
	<9-10> MnSymbolE9
	<10-12> MnSymbolE10
	<12->   MnSymbolE12
}{}
\DeclareFontShape{OMX}{MnSymbolE}{b}{n}{
	<-6>  MnSymbolE-Bold5
	<6-7>  MnSymbolE-Bold6
	<7-8>  MnSymbolE-Bold7
	<8-9>  MnSymbolE-Bold8
	<9-10> MnSymbolE-Bold9
	<10-12> MnSymbolE-Bold10
	<12->   MnSymbolE-Bold12
}{}

\let\llangle\@undefined
\let\rrangle\@undefined
\DeclareMathDelimiter{\llangle}{\mathopen}%
{MnLargeSymbols}{'164}{MnLargeSymbols}{'164}
\DeclareMathDelimiter{\rrangle}{\mathclose}%
{MnLargeSymbols}{'171}{MnLargeSymbols}{'171}
\makeatother
\begin{document}
\title{Abstract Factorization Theorems \\ with Applications to I\-dem\-po\-tent Factorizations}
\author{Laura Cossu}
\address{Institute of Mathematics and Scientific Computing, University of Graz | Heinrichstrasse 36/III, 8010 Graz, Austria}
\email{laura.cossu@uni-graz.at}
\urladdr{https://sites.google.com/view/laura-cossu}

\author{Salvatore Tringali}
\address{School of Mathematical Sciences,
Hebei Normal University | Shijiazhuang, Hebei province, 050024 China}
\email{salvo.tringali@gmail.com}
\urladdr{http://imsc.uni-graz.at/tringali}
%
%\thanks{This manuscript was created on: \today{} at \currenttime}
%
% 06F05: Ordered semigroups and monoids
% 13A05: Divisibility and factorizations in commutative rings
% 13F15: Commutative rings defined by factorization properties
% 13P05: Polynomials, factorization in commutative rings
% 15A23: Factorization of matrices
% 16D70: Structure and classification for modules, bimodules and ideals, direct sum decomposition and cancellation in associative algebras)
% 16E99: Generic for Associative rings and algebras ma in homological methods
% 17C27: Idempotents, Peirce decompositions (in nonassociative rings and algebras)
% 18B35: Preorders, orders, domains and lattices
%
% 20-XX: GROUP THEORY AND GENERALIZATIONS
% 20Mxx: **Semigroups**
% 20M13: Arithmetic theory of monoids
% 20M25: Semigroup rings, multiplicative semigroups of rings
% 47A68: Factorization theory of linear operators


\begin{abstract}
\noindent{}Let $\preceq$ be a preorder on a monoid $H$ with identity $1_H$ and $s$ be an integer $\ge 2$. The $\preceq$-height of an element $x \in H$ is the supremum of the integers $k \ge 1$ for which there is a (strictly) $\preceq$-decreasing sequence $x_1, \ldots, x_k$ of $\preceq$-non-units of $H$ with $x_1 = x$, where $u \in H$ is a $\preceq$-unit if $u \preceq 1_H \preceq u$ and a $\preceq$-non-unit otherwise.  We say $H$ is $\preceq$-artinian if there is no infinite $\preceq$-decreasing sequence of elements of $H$, and strongly $\preceq$-artinian if the $\preceq$-height of each element is finite.

We establish that, if $H$ is $\preceq$-artinian, then each $\preceq$-non-unit $x \in H$ factors through the $\preceq$-irreducibles of degree $s$, where a $\preceq$-irreducible of degree $s$ is a $\preceq$-non-unit $a \in H$ that cannot be written as a product of $s$ or fewer $\preceq$-non-units each of which is (strictly) smaller than $a$ with respect to $\preceq$.
In addition, we show that, if $H$ is strongly $\preceq$-artinian, then $x$ factors through the $\preceq$-quarks of $H$, where a $\preceq$-quark is a $\preceq$-minimal $\preceq$-non-unit. 
In the process, we obtain upper bounds for the length of a shortest fac\-tor\-i\-za\-tion of $x$ into $\preceq$-irreducibles of degree $s$ (resp., $\preceq$-quarks) in terms of its $\preceq$-height.

Next, we specialize these results to the case in which $H$ is the multiplicative submonoid of a ring $R$ formed by the zero divisors of $R$ (and the identity $1_R$) and $a \preceq b$ if and only if the right annihilator of $1_R-b$ is contained in that of $1_R-a$. If $H$ is $\preceq$-artinian (resp., strongly $\preceq$-artinian), then every zero divisor of $R$ factors as a product of $\preceq$-irreducibles of degree $s$ (resp., $\preceq$-quarks); and we prove that, for a variety of right Rickart rings, either the $\preceq$-quarks or the $\preceq$-irreducibles of degree $2$ or $3$ are coprimitive idempotents (an idempotent $e \in R$ is coprimitive if $1_R-e$ is primitive). In the latter case, we also derive sharp upper bounds for the length of a shortest idempotent factorization of a zero divisor $x \in R$ in terms of the $\preceq$-height of $x$ and the uniform dimension of $R_R$.
In particular, we can thus recover and improve on classical theorems of J.\,A.~Erdos (1967), R.\,J.\,H.~Dawlings (1981), and J.~Fountain (1991) on idempotent factorizations in the endomorphism ring of a free module of finite rank over a skew field or a commutative DVD (e.g., we find that every singular $n$-by-$n$ matrix over a commutative \textup{DVD}, with $n\ge 2$, is a product of $2n-2$ or fewer idempotent matrices of rank $n-1$).
\end{abstract}
\maketitle
\thispagestyle{empty}

\section{Introduction}
\label{sec:intro}

In many branches of mathematics, one is often faced with the problem of proving that every ``large element'' of a monoid is a (finite) product of other elements that are regarded as ``elementary building blocks'' since they cannot
be ``broken down into smaller pieces''. One way to make these ideas precise is to combine the language of monoids with the language of preorders, as recently done by the second author in \cite{Tr20(c)}. This is part of a broader program \cite{Tr19, An-Tr18, Fa-Tr18} whose ultimate goal is to enlarge as much as possible the current boundaries of the ``classical theory of factorization'', whose focus is on commutative or ``nearly cancellative'' monoids and where the building blocks are either atoms in the sense of \cite{Cohn68} or irreducibles in the sense of \cite[Definition 2.4]{AnVL96}. We refer the reader to the papers \cite{GeZh22, GeSch21, ChCoGoSm21, Oh-Zh20, GeSch18, Ba-Sm15}, the surveys \cite{GeZh19, Ge16, BaWi13, BaCh11}, and the vol\-umes \cite{FoHoLu13, GeHK06, And97} for an overview of some trends and typical questions in this field. Here, we concentrate on the case where the building blocks are idempotent elements.

More specifically, let $\mathcal{E}(H)$ be the submonoid generated by the i\-dem\-po\-tents of a (commutative or non-commutative) monoid $H$. An element $x \in H$ has an \evid{i\-dem\-po\-tent factorization} if $x \in \mathcal E(H)$, and $H$ is \evid{i\-dem\-po\-tent-generated} if $H = \mathcal E(H)$. Over the years, a great deal of work has been done to characterize  i\-dem\-po\-tent-generated monoids. The first result in this direction is usually credited to J.\,H.~Howie \cite{Ho66}, who proved in 1966 that the monoid of singular transformations of a finite set $X$ is idempotent-generated. One year later, J.\,A.~Erdos showed in \cite{Er68} that, for every positive integer $n$, the monoid of singular $n$-by-$n$ matrices over a field is generated by idempotent matrices.
Afterwards, research in the area has focused on generalizations of Erdos' theorem
from fields to broader classes of rings. In parallel, various authors have studied the problem of bounding the \evid{minimum length} of (i.e., the minimum number of factors in) an i\-dem\-po\-tent factorization of an idempotent-generated element \cite{Laf89, Ha-OM(b), Ha-OM, Ba78}.

Most notably, it was shown by T.\,J.~Laffey \cite[Theorems 1 and 2]{Laf83} that Erdos' theorem carries over to the singular monoid of the ring of $n$-by-$n$ matrices over $R$ when $R$ is either a skew field or a commutative euclidean domain. Later, J.~Fountain \cite[Theorems 4.1 and 4.6]{Fo91} proved that not only the same is true when $R$ is either the integers or a commutative discrete valuation domain (DVD), but also in each of these circumstances the idempotent factors can be taken to be matrices of rank $n-1$. Laffey's results were subsequently generalized by A.~Alahmadi et al.~\cite{Al-Ja-Lam-Le14} to the case where $R$ is a right- and left-quasi-euclidean domain.
On the other hand, R.\,J.\,H.~Dawlings established in \cite{Da81} that any singular endomorphism of an $n$-dimensional vector space is a product of at most $n$ idempotent linear maps of rank $n-1$, so providing a ``quantitative version'' of Erdos' theorem.
We refer the reader to \cite{Ja-Le2019} for a more thorough account of the literature in this area and to \cite{Co-Za19, Co-Za-Za18} for some recent developments.

In the present work, we address both aspects of the study of i\-dem\-po\-tent factorizations (that is, the existence and the bounds) by an approach that is apparently unprecedented. More precisely, assume $\preceq$ is a preorder on a monoid $H$ and let $s$ be an integer $\ge 2$. After in\-tro\-duc\-ing the notions of $\preceq$-non-unit, degree-$s$ $\preceq$-irreducible, $\preceq$-quark, and $\preceq$-height (Definitions \ref{def:2021-06-01-23:10} and \ref{def:2021-06-01-23:16}), we show that, if $\preceq$ is artinian, (i.e., there is no (strictly) $\preceq$-decreasing sequence $x_1, x_2, \ldots$ in $H$), then each $\preceq$-non-unit $x \in H$ factors into a (finite) product of $\preceq$-irreducibles of degree $s$ (Theorem \ref{thm:3.4}). If $\preceq$ is strongly artinian (i.e., the $\preceq$-height of every element is finite) and further conditions are satisfied, then $x$ factors into a product of $\preceq$-quarks (Theorem \ref{thm:3.5}).
In the process, we also obtain upper bounds for the length of a shortest fac\-tor\-i\-za\-tion of $x$ into $\preceq$-irreducible of degree $s$ (resp.,  $\preceq$-quarks) in terms of its $\preceq$-height.

Next, we specialize these abstract results to the case where $H$ is the monoid of zero divisors of a ring $R$, while $\preceq$ is the pre\-order defined by $a \preceq b$ if and only if the right annihilator of $1_R-b $ is contained in that of $1_R-a$ (Sect.~\ref{sec:3.1}). Accordingly, we prove that, for a variety of rings, either the $\preceq$-quarks or the $\preceq$-irreducibles of degree $2$ or $3$ of $H$ are ``coprimitive idempotents'' (Remarks \ref{rem:right-rickart-rings} and \ref{rem:left-sing-implies-sing}, Theorems \ref{thm:rFix-quarks-in-right-rickart} and \ref{thm:5.8}, and Proposition \ref{dge1}\ref{dge1(ii)}). Moreover, we derive sharp upper bounds for the minimum length of an idempotent factorization of an element $x \in H$ in terms of the $\preceq$-height of $x$ and the uniform dimension of $R_R$ (Proposition \ref{prop:5.10} and Theorem \ref{thm:DVD-main-thm}).
In particular, we can recover and improve on some of the classical theorems reviewed in the above. That is, the Erdos-Dawlings-Laffey theorem on singular matrices over a skew field (Corollary \ref{cor:erdos-dawlings}) and Fountain's theorem on singular matrices over a commutative DVD (Corollary \ref{DVD_matrix}). As an aside, we also obtain a characterization of the (multiplicative) monoid generated by the idempotents of a semisimple ring and a tight upper bound for the minimum length of an idempotent factorization of an element from the same (Remark \ref{rem:semisimple-rings}).

Loosely speaking, Theorems \ref{thm:3.4} and \ref{thm:3.5} work as a sort of black box for a wide range of problems. The inputs of the black box are a monoid $H$ and an artinian or strongly artinian preorder $\preceq$ on $H$; the output is the existence of certain factorizations for every ``large element'' of $H$ (cf.~the first paragraph of this section), where an element is taken to be ``large'' if it is not $\preceq$-equivalent to the identity of $H$ (two elements $x, y \in H$ are $\preceq$-equivalent if $x \preceq y \preceq x$).
Here, we test this approach against the study of i\-dem\-po\-tent factorizations; further applications (e.g., to the classical theory of factorization) are discussed in \cite[Sect.~4.1]{Tr20(c)}, while prospects for future research are outlined in Sect.~\ref{sec:6}.

We note in passing that the study of \emph{commutative} rings with non-trivial zero divisors from the point of view of the classical theory of factorization is a topic with a long tradition (see \cite[Sect.~2.4 and Remark 4.5]{An-Tr18} and references therein). However, the classical theory of factorization corresponds, from the point of view of the underlying philosophy of this work, to the case where $H$ is the multiplicative monoid of a ring and the preorder $\preceq$ in the foregoing discussion is the divisibility preorder on $H$ (so that $x \preceq y$ if and only if $y \in HxH$), which has no bearing on the ``non-classical factorizations'' considered herein.

\section{Preliminaries}
\label{sec:2}

Below, we collect some elementary but fundamental results that will often come into play in later sec\-tions, and we review notations and terminology used all through the paper. Further notations and terminology, if not explained when first introduced, are standard or should be clear from context.

\subsection{Generalities}\label{sec:2.1}
We denote by $\mathbb Z$ the (set of) integers, by $\mathbb{N}$ the non-negative integers, and by $\mathbb{N}^+$ the positive integers. For all $a, b \in \allowbreak \mathbb{Z} \cup \{\pm \infty\}$, we let $\llb a, b \rrb := \{x \in \mathbb{Z}\colon a \le x \le b\}$ be the \evid{discrete interval} between $a$ and $b$. Unless noted otherwise, we reserve the letters $m$ and $n$ (with or without subscripts or su\-per\-scripts) for positive integers; and the letters $i$, $j$, $k$, and $s$ for non-negative integers. 

\subsection{Monoids}
\label{sec:monoids}
We take a monoid to be a semigroup with an identity. Unless stated otherwise, monoids will typically be written multiplicatively and need not have any special property (e.g., commutativity). We refer to \cite[Ch.~1]{Ho95} for basic aspects of semigroup theory. 

Let $H$ be a monoid with identity $1_H$. 
A \evid{unit} of $H$ is an element $u \in H$ such that $uv = vu = 1_H$ for a provably unique element $v \in H$, named the \evid{inverse} of $u$ (in $H$) and denoted by $u^{-1}$. An \evid{i\-dem\-po\-tent} of $H$ is an element $e \in H$ such that $e^2 = e$. In particular, an i\-dem\-po\-tent is \emph{proper} if it is not the identity.

An element $a \in H$, on the other hand, is \evid{right-cancellative} (resp., \evid{left-cancellative}) if $xa \ne ya$ (resp., $ax \ne ay$) for all $x, y \in H$ with $x \ne y$; otherwise, $a$ is \evid{right-singular} (resp., \evid{left-singular}). Accordingly, 
$a$ is \evid{cancellative} if it is right- and left-cancellative, and is \evid{singular} if it is right- and left-singular. 

We will write $\mathcal{E}(H)$ for the subsemigroup (in fact, a submonoid) of $H$ generated by the i\-dem\-po\-tents; $H^\#$ for the set made up of the singular elements of $H$ and the identity $1_H$; and $H^\times$ for the set of \evid{units} of $H$. It is readily seen that $H^\#$ is a submonoid and $H^\times$ is a subgroup of $H$, henceforth called the \evid{singular monoid} and the \evid{group of units} of $H$, resp. Similar considerations apply to the left-singular elements of $H$: Together with the identity, they too form a submonoid of $H$, which we refer to as the \evid{monoid of left-singular elements} of $H$ but for which we introduce no special notation. (The case with right-singular elements is completely analogous.) It is straightforward that, if $1_H \ne x \in \mathcal{E}(H)$, then $x$ is \evid{singular}. For, if $x$ factors as a non-empty product $e_1 \cdots e_n$ of i\-dem\-po\-tents of $H$ and we assume without loss of generality that $e_i \ne 1_H$ for each $i \in \llb 1, n \rrb$, then $e_1 x = 1_H x = x 1_H = x e_n$ (i.e., $x$ is singular). Hence $\mathcal{E}(H) \subseteq H^\#$.

We say that $H$ is \evid{von Neumann regular} if, for each $x \in H$, there is $y \in H$ such that $x = xyx$; and \evid{Dedekind-finite} if, for all $x, y \in H$, $xy$ is a unit if and only if at least one (and hence all) of $x$, $y$, and $yx$ is a unit (equivalently, $xy = 1_H$ if and only if $yx = 1_H$). Moreover, we say an element $x \in H$ \emph{factors through} (the elements of) a set $A \subseteq H$ if $x$ can be expressed as a finite product of elements from $A$ (with the usual understanding that an empty product is equal to the identity $1_H$), and a set $X \subseteq H$ factors through (the elements of) $A$ if each element of $X$ factors through $A$.

\subsection{Rings and modules}\label{sec:2.3}
We refer the reader to \cite{An-Fu92} for basic aspects on rings and modules. Most notably, a ring will always mean an associative (commutative or non-commutative) non-zero ring with identity; and a module will mean a \emph{unitary} module, and more precisely a \emph{right} unitary module unless stated otherwise.

Let $R$ be a ring and $M$ be a module over $R$. 
We denote by $0_R$ (resp., by $1_R$) the additive (resp., multiplicative) identity of $R$, by $0_M$ the zero of $M$, and by $R_R$ the \evid{regular right module} over $R$ (i.e., $R$ viewed in the usual way as a right module over itself).  For $X, Y \subseteq M$, $v \in M$, and $A \subseteq R$, we set 
\[
X+Y := \{x+y \colon x \in X, \ y \in Y\} 
\quad\text{and}\quad
vA := \{va \colon a \in A\}.
\]
In particular, we refer to the sumset $X+Y$ as a \evid{direct sum}, and we write $X \oplus Y$ in place of $X+Y$, if $X$ and $Y$ \emph{intersect trivially}, that is, $X \cap Y = \{0_M\}$.
%see, e.g., \cite[Ch.~1, Sect.~5]{An-Fu92}. 
We use $\mathcal M_n(R)$ for the ring of $n$-by-$n$ matrices over $R$ and $\End(M)$ for the \evid{endomorphism ring} of $M$, i.e., the ring obtained by endowing the set of all (right) $R$-linear functions on $M$ with the operations of pointwise addition and functional composition (the latter playing the role of ring multiplication): The additive identity of $\End(M)$ is the func\-tion $M \to M \colon x \mapsto 0_M$, while the multiplicative identity is the identity map $\mathrm{id}_M$ on $M$.

A submodule $N$ of $M$ is \evid{in\-de\-com\-pos\-a\-ble} if $N$ is neither the zero module $\{0_M\}$ nor the direct sum of two non-zero submodules; and is a \evid{direct summand} (of $M$) if $M = N \oplus N'$ for some (not-necessarily-unique) sub\-module $N'$, herein named an \evid{ad\-di\-tive complement} of $N$ (relative to $M$). 
If $N$ is a direct summand of $M$ and $N'$ is an additive complement of $N$, we have a well-defined endo\-morphism $p \in \End(M)$, called the \evid{projection} of $M$ on $N$ along $N'$, that maps a vector $v \in M$ to the u\-nique $x \in N$ such that $v - x \in N'$.

We take the \evid{uniform} (or \evid{Goldie}) \evid{dimension} $\udim(M)$ of $M$ to be the supremum of the set of all integers $k \ge 1$ such that $X_1 \oplus \cdots \oplus X_k \subseteq M$ for some submodules $X_1, \ldots, X_k$ of $M$, with $\sup \emptyset := 0$. Accordingly, we let the \evid{right uniform dimension} of a right ideal $\mathfrak{i}$ of $R$, herein denoted by $\rdim_R(\mathfrak i)$, be the uniform dimension of $\mathfrak i$ viewed as a submodule of $R_R$.
It is seen from \cite[Definition (6.2) and Corollaries (6.6) and (6.10)]{Lam99} that the right u\-ni\-form dimension is \emph{monotone} and \emph{additive}, in the sense that
\begin{equation}\label{equ:r.dim-monotone}
\rdim_R(\mathfrak i) \leq \rdim_R(\mathfrak j), 
\qquad 
\text{for all right ideals } \mathfrak{i}, \mathfrak{j} \subseteq R \text{ with } \mathfrak i \subseteq \mathfrak j
\end{equation}
and
\begin{equation}\label{equ:2021-04-28-23:55}
\rdim_R(\mathfrak i + \mathfrak j) = \rdim_R(\mathfrak i) + \rdim_R(\mathfrak j), 
\qquad 
\text{for all right ideals } \mathfrak{i}, \mathfrak{j} \subseteq R \text{ with } \mathfrak i \cap \mathfrak j = \{0_R\}.
\end{equation}
Given $a \in R$, we write $\rann_R(a)$ for the \evid{right annihilator} and $\lann_R(a)$ for the \evid{left annihilator} of $a$, i.e., 
\[
\rann_R(a) := \{x \in R \colon ax = 0_R\}
\quad\text{and}\quad
\lann_R(a) := \{x \in R \colon xa = 0_R\}.
\]
In addition, we define
\begin{equation}\label{equ:r.fix-in-a-ring}
\rfix_R(a) := \rann_R(1_R-a) = \{x \in R \colon ax = x\}.
\end{equation}
As a matter of fact, $\rann_R(a)$ is a right ideal of $R$, and hence so also is $\rfix_R(a)$. Moreover, it is obvious that $ax = 0_R \neq x$ for every non-zero $x \in \rann_R(a)$, which shows that 
\begin{equation}\label{equ:2021-06-01-14:47}
\rann_R(a) + \rfix_R(a) = \rann_R(a) \oplus \rfix_R(a).
\end{equation}
We call $R$ a \textsf{right Rickart} ring if $\rann_R(a)$ is a direct summand of $R_R$ for every $a \in R$. As will be seen in Sect.~\ref{subsec:characterizations}, this class of rings is of utmost importance in the applications of the abstract theorems of Sect.~\ref{sec:3}. 

We say that $R$ is a \evid{von Neumann regular} (resp., \evid{Dedekind-finite}) ring if so is the multiplicative monoid $(R, \cdot\,)$ of $R$. By a unit of $R$ we always understand a \emph{multiplicative} unit, that is, a unit of $(R, \cdot\,)$. The same goes for i\-dem\-po\-tents and singular, left-singular, and right-singular elements.
Accordingly, we define $R^\# := (R, \cdot\,)^\#$ and $R^\times := (R, \cdot\,)^\times$, and we denote by $\mathcal E(R)$ the sub\-sem\-i\-group of $(R, \cdot\,)$ generated by the i\-dem\-po\-tents of $R$. Note that an element $a \in R$ is left-singular (resp., right-singular) if and only if $a$ is a right (resp., left) zero divisor of $R$. Accordingly, we will simply refer to $R^\#$ and $R^\times$, resp., as the \evid{monoid of zero divisors} and the \evid{group of units} of the ring $R$. Likewise, we will talk of the ``monoid of left (resp., right) zero divisors'' of $R$ to mean the monoid of left-singular (resp., right-singular) elements of $(R, \cdot\,)$.

In particular, we let an idempotent $e \in R$ be \textsf{coprimitive} if $\rann(e)$ is an indecomposable sub\-mod\-ule of $R_R$. The notion is left-right symmetric. It is a basic fact (see, e.g., 
\cite[p.~48]{Ja64}) that
%\cite[Sect.~21, p.~308]{Lam01}) that 
\begin{equation}\label{equ:direct-sum-decomposition-idemp}
R = eR \oplus (1_R-e)R = Re \oplus R(1_R-e),
\end{equation}
and this implies at once that 
\begin{equation}\label{equ:ann&fix-of-idempotent}
\rann_R(e) = \rfix_R(1_R-e) = (1_R-e)R
\quad\text{and}\quad
\lann_R(e) = R(1_R-e). 
\end{equation}
Consequently, we conclude from the equivalence between conditions (1) and (1)$'$ in \cite[Proposition (21.8)]{Lam01} that $e$ is coprimitive if and only if $\lann_R(e)$ is indecomposable as a \emph{left} module over $R$.

Lastly, we let a \evid{Peirce basis} of $R$ be a non-empty, finite tuple $(e_1, \ldots, e_n)$ of pairwise or\-thog\-o\-nal i\-dem\-po\-tents of $R$ such that $1_R = e_1 + \cdots + e_n$, where $x, y \in R$ are \evid{orthogonal} if $xy = yx = 0_R$. 

The following result provides a characterization of i\-dem\-po\-tents that, while elementary, is central for this paper (note that we henceforth drop the sub\-script `$R\,$' from the notation `$\rann_R$', `$\lann_R$', and `$\rfix_R$' when the ring $R$ is clear from context).

\begin{proposition}\label{prop:caratterizzazione-idemps}
Let $R$ be a ring. The following are equivalent for an element $a \in R$:
\begin{enumerate}[label=\textup{(\alph{*})}]
\item\label{prop:caratterizzazione-idemps(a)} $a$ is an i\-dem\-po\-tent of $R$.
%\item\label{prop:caratterizzazione-idemps(b)} $R = \rann(a) + \rfix(a)$.
\item\label{prop:caratterizzazione-idemps(c)} $R = \rann(a) \oplus \rfix(a)$.
\item\label{prop:caratterizzazione-idemps(d)} $aR = \rfix(a)$.
\end{enumerate}
\end{proposition}

\begin{proof}
\ref{prop:caratterizzazione-idemps(a)} $\Rightarrow$ \ref{prop:caratterizzazione-idemps(c)}: 
Let $x \in R$. Since $a$ is i\-dem\-po\-tent, we have $ax = a^2x$ and hence $a(x-ax) = 0_R$; that is, $ax \in \rfix(a)$ and $x - ax \in \rann(a)$. It follows that $x = (x-ax) + ax \in \rann(a) + \rfix(a)$, which shows in turn that $R \subseteq \rann(a) + \rfix(a)$. The opposite inclusion is obvious, so by Eq.~\eqref{equ:2021-06-01-14:47} we are done.

\vskip 0.05cm

\ref{prop:caratterizzazione-idemps(c)} $\Rightarrow$ \ref{prop:caratterizzazione-idemps(d)}: Let $b \in R$. By the hypothesis that $R = \rann(a) \oplus \rfix(a)$, we can write $b = x+y$ for some $x \in \rann(a)$ and $y \in \rfix(a)$. This implies $ab = ax + ay = y \in \rfix(a)$ and hence $aR \subseteq \rfix(a)$. On the other hand, it is clear that $\rfix(a) \subseteq aR$, because $z = az$ for each $z \in \rfix(a)$. So, $aR = \rfix(a)$.

\vskip 0.05cm

\ref{prop:caratterizzazione-idemps(d)} $\Rightarrow$ \ref{prop:caratterizzazione-idemps(a)}: 
Since $aR = \rfix(a)$, we have $a \in \rfix(a)$ and hence $a^2 = a$ (namely, $a$ is i\-dem\-po\-tent).
\end{proof}

The next proposition, on the other hand, puts together a few more basic properties of idempotents that we will often come in handy in Sects.~\ref{sec:4} and \ref{subsec:characterizations}.

\begin{proposition}\label{prop:2.2}
Let $R$ be a ring and $e, f \in R$ be idempotents. The following hold:
\begin{enumerate}[label=\textup{(\roman{*})}]
\item\label{prop:2.2(i)} If $R = \mathfrak i_1 \oplus \cdots \oplus \mathfrak i_n$ for certain right ideals $\mathfrak i_1, \ldots, \mathfrak i_n \subseteq R$, then there is a uniquely determined Peirce basis $(e_1, \ldots, e_n)$ of $R$ such that $\mathfrak i_1 = e_1 R, \ldots, \mathfrak i_n = e_n R$.
\item\label{prop:2.2(ii)} $eR = \mathfrak i_1 \oplus \mathfrak i_2$, for some right ideals $\mathfrak i_1, \mathfrak i_2 \subseteq R$, if and only if there exist orthogonal i\-dem\-po\-tents $e_1, e_2 \in R$ with $e = e_1 + e_2$ such that $\mathfrak i_1 = e_1 R$ and $\mathfrak i_2 = e_2R$.
\item\label{prop:2.2(iv)} If $eR$ is contained in $fR$, then $fR = eR \oplus f_0R$ for some idempotent $f_0 \in R$.
\item\label{prop:2.2(v)} If $ef = fe = 0_R$ and $z \in eRf \cup fRe$, then $e+z$ is an idempotent of $R$.
\end{enumerate}
\end{proposition}

\begin{proof}
For (the left analogue of) item \ref{prop:2.2(i)},
see Proposition 7.2 in \cite{An-Fu92},
and for item \ref{prop:2.2(ii)}, see Proposition (21.8) in \cite{Lam01} (the case $e = 0_R$ is trivial).

\vskip 0.05cm

\ref{prop:2.2(iv)}: Since $eR \subseteq fR$, we have from \cite[Lemma 4(i)]{Ma60} that there exists an idempotent $e_0 \in R$ such that $e_0R = eR$ and $e_0 = e_0f = fe_0$ (in fact, we may take $e_0 := ef$). Set $f_0 := f - e_0 = f(1_R - e_0)$. Then 
\[
f_0^2 = f^2 - fe_0 - e_0 f + e_0^2 = f - e_0 = f_0,
\]
namely, $f_0$ is idempotent. 
It thus remains to see that $fR=e_0R\oplus f_0R$. For, note first that $e_0R$ and $f_0R$ intersect trivially. Indeed, if $ x = e_0r = f_0s$ for some $r, s \in R$, then 
\[
x = e_0^2\, r = e_0 x = e_0 f_0 s = e_0 (f - e_0) s = (e_0 f - e_0^2) s = (e_0 - e_0)s = 0_R. 
\]
Moreover, $e_0R + f_0R = fe_0R\oplus f(1_R-e_0)R\subseteq fR$, and on the other hand, $fR \subseteq e_0R + f_0R$ because $f = e_0+f_0$. So putting it all together, we find $fR = e_0R \oplus f_0R$ (as wished).

\vskip 0.05cm

\ref{prop:2.2(v)}: If $z \in eRf \cup fRe$ and $ef = fe = 0_R$, then it is clear that $z^2 = 0_R$ and $ez + ze = z$, which shows in turn that $(e+z)^2 = e^2 + ez + ze + z^2 = e+z$ (i.e., $z$ is an idempotent of $R$).
\end{proof}

\section{Preorders and abstract factorization theorems}
%and the \texorpdfstring{$\rfix$}{}-preorder}
\label{sec:3}

To start with, we recall that a \evid{preorder} on a set $X$ is a reflexive and transitive relation $\preceq$ on $X$. In particular, we write $x \prec y$ to mean that $x \preceq y$ and $y \not\preceq x$, and we recall from the introduction that $\preceq$ is an \evid{artinian} preorder if there is no sequence $(x_k)_{k \ge 0}$ of elements of $X$ with $x_{k+1} \prec x_k$ for each $k \in \mathbb{N}$. Here, we are mostly interested in the interplay between preorders and monoids. This leads to the following.

\begin{definition}\label{def:2021-06-01-23:10}
Let $\preceq$ be a preorder on a monoid $H$. An element $u \in H$ is a \evid{$\preceq$-unit} (of $H$) if $u \preceq 1_H \preceq u$, and is a \evid{$\preceq$-non-unit} if it is not a $\preceq$-unit. 
Accordingly, we say that a $\preceq$-non-unit $a \in H$ is 
\begin{itemize}
\item a \textsf{$\preceq$-irreducible of degree $s$} or {\sf degree-$s$ $\preceq$-irreducible} (of $H$), for some integer $s \ge 2$, if for every $k \in \llb 2, s \rrb$ there exist no $\preceq$-non-units $b_1, \ldots, b_k$ with $b_1 \prec a, \ldots, b_k \prec a$ such that $a = b_1 \cdots b_k$;
\item a \evid{$\preceq$-quark} (of $H$) if there does not exist any $\preceq$-non-unit $b$ with $b \prec a$. 
\end{itemize}
In particular, we will simply refer to a $\preceq$-irreducible of degree $2$ as a {\sf $\preceq$-irreducible} (of $H$); and say that $H$ is \evid{$\preceq$-factorable} if every $\preceq$-non-unit of $H$ factors as a product of $\preceq$-irreducibles.
\end{definition}

It is straightforward that, if $\preceq$ is a preorder on a monoid $H$, then a $\preceq$-irreducible of degree $s \ge 2$ is also a $\preceq$-irreducible of degree $k$ for every $k \in \llb 2, s \rrb$, and each $\preceq$-quark is a degree-$s$ $\preceq$-irreducible  for every $s \ge 2$ (for a partial converse to this latter statement, see Theorem \ref{thm:3.5} below).

\begin{definition}\label{def:2021-06-01-23:16}
Given a preorder $\preceq$ on a monoid $H$ and an element $x \in H$, we denote by $\hgt_\preceq^H(x)$ the supremum of the set of all integers $n \ge 1$ for which there are $\preceq$-non-units $x_1, \ldots, x_n \in H$ with $x_1 = x$ and $x_{k+1} \prec x_k$ for each $k \in \llb 1, n-1 \rrb$, with $\sup \emptyset := 0$. We call $\hgt_\preceq^H(x)$ the \evid{$\preceq$-height} of $x$ (relative to $H$); and we say $\preceq$ is a \evid{strongly artinian} preorder (on $H$), or $H$ is a \evid{strongly $\preceq$-artinian} monoid, if the $\preceq$-height of every element is finite. We will refer to the map $H \to \mathbb{N} \cup \{\infty\} \colon y \mapsto \hgt_\preceq^H(y)$ as the \evid{$\preceq$-height} (\evid{function}) of $H$, and write $\hgt(x)$ in place of $\hgt_{\preceq}^H(x)$ if no confusion can arise.
\end{definition}

The following observations, although elementary, will prove useful in several occasions.

\begin{remarks}\label{rem:height}
\begin{enumerate*}[label=\textup{(\arabic{*})}]
\item\label{rem:height(1)} Let $K$ be a submonoid of a monoid $H$ and $\preceq_K$ be the restriction to $K$ of a preorder $\preceq$ on $H$. If $1_H$ is the only $\preceq$-unit of $H$, then $1_H$ is also the only $\preceq_K$-unit of $K$ (recall that $1_K = 1_H$). In particular, the $\preceq_K$-height of an element $a \in K$ is no larger than the $\preceq$-height of $a$ (relative to $H$), which shows in turn that, if $H$ is strongly $\preceq$-artinian, then $K$ is strongly $\preceq_K$-artinian. \\[0.05cm]

\indent{}\item\label{rem:height(2)} Let $\preceq$ be a preorder on a monoid $H$, and for a fixed $a \in H$ denote by $\Lambda(a)$ the set of all integers $n \ge 1$ for which there exist $\preceq$-non-units $x_1, \ldots, x_n \in H$ with $x_1 = a$ ed $x_{k+1} \prec x_k$ for each $k \in \llb 1, n-1 \rrb$ (so, if $a$ is a $\preceq$-unit, then $\Lambda(a) = \emptyset$). It is immediate that $\Lambda(a) = \llb 1, \hgt(a) \rrb$. For, note that, if an integer $n$ is in $\llb 1, \hgt(a) \rrb$ but not in $\Lambda(a)$, then $\Lambda(a) \cap \llb n, \infty \rrb = \emptyset$ and hence $\hgt(a) = \sup \Lambda(a) \le n-1 < \hgt(a)$ (which is absurd).
\end{enumerate*}
\end{remarks}

With the exception of $\preceq$-irreducibles of degree $s \ge 3$, the notions introduced in Definitions \ref{def:2021-06-01-23:10} and \ref{def:2021-06-01-23:16} were first considered in \cite[Sect.~3]{Tr20(c)}.
Their significance is linked to the wide range of applications of the next results, which generalize Theorem 3.10 and Proposition 3.14 of \cite{Tr20(c)} from the case where $s = 2$.

\begin{theorem}\label{thm:3.4}
Let $\preceq$ be an artinian preorder on a monoid $H$ and $s$ be an integer $\ge 2$. Then every $\preceq$-non-unit $x \in H$ factors as a non-empty, finite product of $s^{\hgt(x)-1}$ or fewer $\preceq$-irreducibles of degree $s$.
\end{theorem}

\begin{proof}
We divide the proof into two parts: In \textsc{Part 1}, we prove that every $\preceq$-non-unit $x \in H$ factors as a finite product of $\preceq$-irreducibles of degree $s$; and in \textsc{Part 2}, we establish that the minimum length of such a fac\-tor\-i\-za\-tion is no larger than $s^{\hgt(x)-1}$.

\vskip 0.05cm 

\textsc{Part 1:} Let $\Omega$ be the set of $\preceq$-non-units of $H$ that do not factor as a finite product of $\preceq$-irreducibles of degree $s$, and suppose for the sake of a contradiction that $\Omega$ is non-empty. Since $\preceq$ is an artinian preorder, it is then a routine matter to show that $\Omega$ has a $\preceq$-minimal element $\bar{x}$, meaning that there exists no $x \in \Omega$ with $x \prec \bar{x}$ (see \cite[Remark 3.9(3)]{Tr20(c)} for further details). Clearly, $\bar{x}$ is a $\preceq$-non-unit but not a $\preceq$-irreducible of degree $s$. So, there are $k \in \llbracket 2, s \rrbracket$ and $\preceq$-non-units $x_1, \ldots, x_k \in H$ such that $x_i \prec \bar{x}$ for every $i \in \llb 1, k \rrb$ and $\bar{x} = x_1 \cdots x_k$. On the other hand, the $\preceq$-minimality of $\bar{x}$ as an element of $\Omega$ yields $x_1, \ldots, x_k \notin \Omega$. Thus, each of $x_1, \ldots, x_k$ factors as a product of $\preceq$-irreducibles of degree $s$, which however implies that the same holds for the product $x_1 \cdots x_k$ and hence contradicts that $x_1 \cdots x_k = \bar{x} \in \Omega$.

\vskip 0.05cm 

\textsc{Part 2:} Pick a $\preceq$-non-unit $x \in H$, denote by $\ell(x)$ the minimum length of a fac\-tor\-i\-za\-tion of $x$ into $\preceq$-irreducibles of degree $s$ (observe that, by the first part of the proof, $\ell(x)$ is a well-defined positive integer), and set $N := \hgt(x)$. We need to prove that $\ell(x) \le s^{N-1}$. 
If $N = \infty$, the con\-clu\-sion is immediate. So, suppose $1 \le N < \infty$ (note that the $\preceq$-height of an element $u \in H$ is zero if and only if $u$ is a $\preceq$-unit). We proceed by induction on $N$.

If $N = 1$, then $x$ is a $\preceq$-quark and hence $\ell(x) = 1 = s^{N-1}$, since a $\preceq$-quark is a degree-$t$ $\preceq$-irreducible for every integer $t \ge 2$. Thus, let $N \ge 2$ and assume inductively that every $\preceq$-non-unit of $\preceq$-height $h \le N-1$ factors as a product of $s^{h-1}$ or fewer $\preceq$-irreducibles of degree $s$. If $x$ is a $\preceq$-irreducible of degree $s$, then $\ell(x) = 1 \le s^{N-1}$ and we are done. Otherwise, there are $k \in \llb 2, s \rrb$ and $\preceq$-non-units $x_1, \ldots, x_k\in H$ such that $x=x_1\cdots x_k$ and $x_i\prec x$ for every $i \in \llb 2, k \rrb$. In particular, we have $h_i := \hgt(x_i) \le N-1$ for each $i \in \llb 1, k \rrb$, which, by the inductive hypothesis, tells us that $x_i$ is a product of $\ell_i$ degree-$s$ $\preceq$-irreducibles for some $\ell_i \in \llb 1, s^{h_i-1} \rrb$. Consequently, $\ell(x) \le \ell_1 + \cdots + \ell_k \leq k s^{N-2}\leq s^{N-1}$, as wished.
\end{proof}

The upper bound derived in Theorem \ref{thm:3.4} for the minimum length of a factorization of a fixed element of $H$ into a product of $\preceq$-irreducibles of degree $s$ is rather weak and can be greatly improved under further assumptions on the preorder $\preceq$, as in the next result.

\begin{theorem}\label{thm:3.5}
Let $\preceq$ be a strongly artinian preorder on a monoid $H$ and $s$ be an integer $\ge 2$ such that, for every $x \in H$ that
is neither a $\preceq$-unit nor a $\preceq$-quark, there exists $k \in \llb 2, s \rrb$ such that $x =y_1 \cdots y_k$ for some  $\preceq$-non-units $y_1, \ldots, y_k \in H$, with the additional property that $y_i\preceq x$ for each $i \in \llb 1, k \rrb$ and $\hgt(y_1) +  \allowbreak \cdots + \hgt(y_k) \leq \allowbreak \hgt(x) + k - 2$. Then the following hold:
\begin{enumerate}[label=\textup{(\roman{*})}]
\item\label{thm:3.5(i)} The preorder $\preceq$ is artinian, the $\preceq$-non-units of $H$ factor through the $\preceq$-irreducibles of degree $s$, and every $\preceq$-irreducible of degree $s$ is a $\preceq$-quark.
\item\label{thm:3.5(ii)} Each $\preceq$-non-unit $x\in H$ is a product of $(s-1)\hgt(x)-(s-2)$ or fewer $\preceq$-quarks.
\end{enumerate}
\end{theorem}

\begin{proof}
\ref{thm:3.5(i)}: By the assumption that $H$ is strongly $\preceq$-artinian, it is clear that the function $\lambda \colon H \to \mathbb{N} \colon x \mapsto \hgt(x)$ is well defined. On the other hand, it is obvious from the definition of the $\preceq$-height that $\lambda(x) < \lambda(y)$ whenever $x \prec y$. Consequently, we see that $\preceq$ is an artinian preorder, or else there would exist a sequence $(N_k)_{k\ge 0}$ of non-negative integers such that $N_{k+1} < N_k$ for each $k \in \mathbb{N}$ (a contradiction). We thus get from Theorem \ref{thm:3.4} that every $\preceq$-non-unit of $H$ factors as a (finite) product of $\preceq$-irreducibles of degree $s$, and it only remains to show that each $\preceq$-irreducible of degree $s$ is in fact a $\preceq$-quark. 

Let $x \in H$ be neither a $\preceq$-unit nor a $\preceq$-quark; we need to check that $x$ is not a $\preceq$-irreducible of degree $s$. By hypothesis, there exist $k \in \llb 2,s \rrb$ and $\preceq$-non-units $y_1, \ldots, y_k \in H$ with $y_1 \preceq x, \ldots, y_k \preceq x$ such that $
x = y_1 \cdots y_k$ and $\hgt(y_1) + \cdots + \hgt(y_k) \leq \hgt(x)+k-2$.
If $x\preceq y_j$ for some $j$, then $\hgt(y_j)=\hgt(x)$ and 
\[
k-1 \leq (\hgt(y_1) + \cdots + \hgt(y_k)) - \hgt(y_j) \leq \hgt(x) - \hgt(y_j) + k-2 = k-2,
\]
which is of course impossible. (Recall that the $\preceq$-height of an element $u \in H$ is zero if and only if $u$ is a $\preceq$-unit.)  Therefore, $y_1 \prec x, \ldots, y_k \prec x$ and $x$ is not a $\preceq$-irreducible of degree $s$ (as wished).

\vskip 0.05cm

\ref{thm:3.5(ii)}: Fix a $\preceq$-non-unit $x \in H$, set $N := \hgt(x)$, and denote by $q(y)$ the minimum length of a fac\-tor\-i\-za\-tion of a $\preceq$-non-unit $y \in H$ into $\preceq$-quarks (note that, by item \ref{thm:3.5(i)}, $q(x)$ is a well-defined positive integer and $1 \le N < \infty$). We need to prove $q(x) \le (s-1)N - (s-2)$, and we proceed by induction on $N$.

If $N = 1$, then $x$ is a $\preceq$-quark and the conclusion is trivial. So, assume $N \ge 2$ and suppose inductively that $q(y) \le (s-1)\hgt(y)-(s-2)$ for every $\preceq$-non-unit $y \in H$ of $\preceq$-height $\le N-1$. Since $x$ is neither a $\preceq$-unit nor a $\preceq$-quark, there exists $k \in \llb 1, s \rrb$ such that $x = y_1 \cdots y_k$ for some $\preceq$-non-units $y_1, \ldots, y_k \in H$ with $y_1 \preceq x, \ldots, y_k \preceq x$ and $\hgt(y_1) + \cdots + \hgt(y_k) \leq N+k-2$.

As in the proof of item \ref{thm:3.5(i)}, it follows that $1 \le \hgt(y_i) \le N-1$ for each $i \in \llb 1,k \rrb$, which, by the inductive hypothesis, implies that $y_i$ factors as a product of $(s-1) \hgt(y_i) - (s-2)$ or fewer $\preceq$-quarks. Thus
\[q(x) \leq \sum_{i=1}^k q(y_i) \le \sum_{i=1}^k \bigl((s-1)\hgt(y_i) -(s-2)\bigr) = (s-1)\sum_{i=1}^k \hgt(y_i) - k(s-2),\]
which, using that $k \leq s$, yields $q(x) \le (s-1)(N+k-2)-k(s-2) \le (s-1)N-(s-2)$, as wished.
\end{proof}
While the artinianity of a preorder $\preceq$ on a monoid $H$ is sufficient for each $\preceq$-non-unit of $H$ to factor as a product of $\preceq$-irreducibles (Theorem \ref{thm:3.4}), the same condition is far from being necessary. For, let $H$ be the multiplicative monoid of the non-zero elements of the integral domain constructed by A.~Grams in \cite[Sect.~1]{Gra74} and $\preceq$ be the divisibility preorder on $H$ (see the comments at the end of \cite[Sect.~3]{Tr20(c)} for further details). This raises the question of whether the previous results can be further generalized to the point of proving a (non-trivial) characterization of when the monoid $H$ is $\preceq$-factorable. Moreover, it points out how the characterization of $\preceq$-irreducibles (of a fixed degree) and $\preceq$-quarks is an interesting problem in its own right, independent of the artinianity of $\preceq$.
\subsection{The r.\textsc{F}ix-preorder}
\label{sec:3.1}
Theorems \ref{thm:3.4} and \ref{thm:3.5} lie at the heart of our approach to the study of i\-dem\-po\-tent factorizations. Another essential ingredient in this direction is provided by the following.

\begin{definition}\label{def:r.fix-preorder}
Given a monoid $H$ and an element $a \in H$, we set $\rfix_H(a) := \{x \in H \colon ax = x\}$ and let the \evid{$\rfix$-preorder on $H$} be the binary relation $\preceq$ on $H$ defined by $b \preceq c$ if and only if $\rfix_H(c) \subseteq \rfix_H(b)$.
\end{definition}

In the language of \cite[Example 3.3]{Tr20(c)}, the $\rfix$-preorder on a monoid $H$ is the dual of the pullback of the inclusion order on the power set $\mathcal P(H)$ of $H$ through the function $\phi \colon H \to \mathcal P(H) \colon a \mapsto \rfix_H(a)$; in particular, it is a preorder (as suggested by the name). 

\begin{definition}
A submonoid $K$ of a monoid $H$ is $\rfix_H$\evid{-artinian} (resp., \evid{strongly} $\rfix_H$\evid{-artinian}) if the restriction $\preceq_K$ to $K$ of the $\rfix$-preorder on $H$ is artinian (resp., strongly artinian). Moreover, we refer to an element $u \in K$ as an $\rfix_H$\evid{-unit} (resp., $\rfix_H$\evid{-non-unit}) of $K$ if $u$ is a $\preceq_K$-unit (resp., a $\preceq_K$-non-unit). We set a degree-$s$ $\rfix_H$\evid{-irreducible} (resp., an $\rfix_H$\evid{-quark}) of $K$ to be a degree-$s$ $\preceq_K$-irreducible (resp., a $\preceq_K$-quark) of $K$, and call the $\preceq_K$-height of $K$ the \evid{$\rfix_R$-height} of $K$ (see Definition~\ref{def:2021-06-01-23:16}).
\end{definition}

Note that, given a monoid $H$ and a submonoid $K$ of $H$, the restriction to $K$ of the $\rfix$-preorder on $H$ is not, in general, the $\rfix$-preorder on $K$ (the idea is that, in many situations, we can exploit the preorder defined on the larger monoid to probe the smaller one).

\begin{corollary}\label{cor:2021-06-02-15:38}
Let $H$ be a monoid and $K$ be an $\rfix_H$-artinian submonoid of $H$. Then each $a \in K$ factors as a finite product of degree-$s$ $\rfix_H$-irreducibles of $K$, for every integer $s \ge 2$.
\end{corollary}

\begin{proof}
It is clear that $\rfix_H(x) = H$, for some $x \in K$, if and only if $x = 1_K = 1_H$. Consequently, $1_H$ is the only $\rfix_H$-unit of $K$ and the conclusion follows at once from Theorem \ref{thm:3.4} (applied to $K$), upon considering that $1_H$ is an empty (and hence finite) product of elements from any subset of $H$.
\end{proof}

The next proposition puts together some basic properties of the $\rfix$-preorder that, in addition to being of independent interest, will be useful to deal with endomorphism rings in Sect.~\ref{sec:5.2-endo-rings}.

\begin{proposition}\label{prop:conjugation}
Let $H$ be a submonoid of the multiplicative monoid of a ring $R$ such that $R^\times H R^\times = H$ and let $\preceq$ be the $\rfix_R$-preorder on $H$. Given $a \in R$ and $u \in R^\times$, the following hold:

\begin{enumerate}[label=\textup{(\roman{*})}]
\item\label{prop:conjugation(i)} $\rann_R(u^{-1} au) = u^{-1} \rann_R(a) \, u$ and hence $\rfix_R(u^{-1} a u) = u^{-1} \rfix_R(a) \, u$.
\item\label{prop:conjugation(ii)} $a$ is a left zero divisor of $R$ if and only if so is $u^{-1} a u$.
\item\label{prop:conjugation(iii)} $a$ is an idempotent \textup{(}resp., a coprimitive idempotent\textup{)} of $R$ if and only if so is $u^{-1} a u$.
\item\label{prop:conjugation(iv)} $a$ is an $\rfix_R$-irreducible of $H$ of some degree $s \ge 2$ if and only so is $u^{-1} a u$.
\end{enumerate}
\end{proposition}

\begin{proof}
\ref{prop:conjugation(i)}: Since $a = u(u^{-1} a u) u^{-1}$, it is enough to check that $u^{-1} \rann(a) \, u \subseteq \rann(u^{-1} a u)$, which is rather easy.
In fact, we have $u^{-1} a u(u^{-1} x u) = ax = 0_R$ and hence $u^{-1} x u \in \rann(u^{-1} a u)$ for all $x \in \allowbreak \rann(a)$. Thus $u^{-1} \rann(a) \, u \subseteq \rann(u^{-1} a u)$.

As for the second part, it is enough to note that $1_R - u^{-1} a u = u^{-1} (1_R - a) u$, so we derive from the first part that $\rfix_R(u^{-1}a u) = \rann(1_R - u^{-1} a u) = u^{-1} \rann(1_R - a) \, u = u^{-1} \rfix_R(a) \, u$.

\vskip 0.05cm

\ref{prop:conjugation(ii)}: This is clear from item \ref{prop:conjugation(i)}, since $b \in R$ is a left zero divisor if and only if $\rann(b) \ne \{0_R\}$.

\vskip 0.05cm

\ref{prop:conjugation(iii)}: The first part is obvious, since $a^2 = a$ implies $(u^{-1} a u)^2 = u^{-1} a^2 u = u^{-1} a u$. For the second part, it suffices to check that, if $a$ is a coprimitive idempotent, then so also is $u^{-1} a u$ (cf.~the proof of item \ref{prop:conjugation(i)}). 

For, suppose $1_R-u^{-1} a u = e + f$ for some orthogonal idempotents $e, f \in R$. Then $1_R-a = ueu^{-1} + \allowbreak ufu^{-1}$; and since $ueu^{-1} ufu^{-1} = uef u^{-1} = 0_R$ and similarly $ ufu^{-1} ueu^{-1} = 0_R$, we gather from the first part that $1_R-a$ is the sum of two orthogonal idempotents. Consequently, we see that $a$ is a coprimitive i\-dem\-po\-tent of $R$ only if so is $u^{-1} a u$.

\vskip 0.05cm

\ref{prop:conjugation(iv)}: Suppose $u^{-1} a u$ is not a degree-$s$ $\rfix_R$-irreducible of $H$ for a certain $s \ge 2$; as in the proof of item \ref{prop:conjugation(i)}, it will be enough to show that neither is $a$. Since the only $\rfix_R$-unit of $H$ is the identity $1_R$ and, on the other hand, $u^{-1} b u = 1_R$, for some $b \in R$, if and only if $a = 1_R$, we may assume that $u^{-1} a u$ is an $\rfix_R$-non-unit of $H$. There then exist $k \in \llb 2, s \rrb$ and $\rfix_R$-non-units $a_1, \ldots, a_k$ of $H$ such that $u^{-1}a u = \allowbreak a_1 \cdots a_k$ and $\rfix_R(u^{-1}au) \subsetneq \rfix_R(a_i)$ for each $i \in \llb 1, k \rrb$. It is therefore immediate that $a = \allowbreak u\,a_1 u^{-1} \cdots u\,a_k u^{-1}$. In addition, $u\,a_i u^{-1}$ is a non-identity element of $H$ for each $i \in \llb 1, k \rrb$, by the assumption that $R^\times H R^\times = H$; and it is straightforward from \ref{prop:conjugation(i)} that $\rfix_R(a) \subsetneq \rfix_R(u\, a_i u^{-1}) \subsetneq R$ (note that, if $X \subsetneq Y \subsetneq R$, then $v X w \subsetneq v Y w \subsetneq R$ for all $v, w \in R^\times$). As a result, $a$ is not an $\rfix_R$-irreducible of $H$, because it factors into a product of $s$ or fewer $\rfix_R$-non-units of $H$ each of which is (strictly) smaller than $a$ with respect to the $\rfix_R$-preorder on $H$.
\end{proof}

In the end, Corollary \ref{cor:2021-06-02-15:38} shows that, for an $\rfix_H$-artinian submonoid $K$ of a monoid $H$ to be i\-dem\-po\-tent-generated, it is sufficient to prove that every $\rfix_H$-irreducible of some degree $s$ of $K$ is i\-dem\-po\-tent. We will see this principle at work in Sect.~\ref{subsec:characterizations}. Necessary conditions, on the other hand, are discussed in the following remark. It turns out that the monoid of zero divisors of a ring $R$ can only be idempotent-generated if the multiplicative monoid of $R$ is \evid{directly irreducible}, that is, does not factor into the direct product of two non-trivial monoids.

\begin{remark}
\label{rem:idem-factorizations-in-direct-products}
Let $H := \prod_{i \in I} H_i$ be the \evid{direct product} of a family $(H_i)_{i \in I}$ of monoids indexed by the non-empty set $I$ (we refer to \cite[pp.~5--6]{Ho95} for notation and terminology), and let $S := \bigcup_{i \in I} H_i$. Given $j \in I$ and $a \in H_j$, we will write $\delta_j^{(a)}$ for the function $I \to S$ that maps $j$ to $a$ and an index $i \ne j$ to $1_{H_i}$.

It is fairly clear that an idempotent of $H$ is a function $e \colon I \to S$ such that $e(i)$ is an idempotent of $H_i$ for all $i \in I$. In addition, every $f \in \prod_{i \in I} H_i^\#$ is also in $H^\#$: This is obvious if $f(i) = 1_{H_i}$ for each $i \in I$. Otherwise, $f(j)$ is a singular element of $H_j$ for some $j \in I$, i.e., there exist $x, y, z, w \in H_j$ with $x \ne y$ and $z \ne w$ such that $x f(j) = y f(j)$ and $f(j)\, z = f(j)\, w$. So, we find $\delta_j^{(x)} f = \delta_j^{(y)} f$ and $f\delta_j^{(z)} = f\delta_j^{(w)}$, which shows that $f$ is in $H^\#$, because $\delta_j^{(x)} \ne \delta_j^{(y)}$ and $\delta_j^{(z)} \ne \delta_j^{(w)}$.
%\end{enumerate*}

Since an element that factors into a product of proper idempotents is necessarily singular (as already mentioned in Sect.~\ref{sec:monoids}), it thus follows that 
\begin{equation}\label{equ:2021-06-17-17:44}
\mathcal E(H) \subseteq \prod_{i \in I} \mathcal E(H_i) \subseteq \prod_{i \in I} H_i^\# \subseteq H^\#,
\end{equation}
with $\prod_{i \in I} \mathcal E(H_i) = \prod_{i \in I} H_i^\#$ if and only if $\mathcal E(H_i) = H_i^\#$ for each $i \in I$. We will look for conditions under which each of the left-most and right-most inclusions in Eq.~\eqref{equ:2021-06-17-17:44} holds as an equality (trivially, this is the case if $H$ is the multiplicative monoid of a domain $R$, for then the only elements of $H^\#$ are $0_R$ and $1_R$).

To start with, we recall from \cite[Definitions 1.1]{Hig94} that the \evid{idempotent depth} of an idempotent-generated monoid $M$, herein denoted by $\Delta(M)$, is the infimum of the set of integers $k \ge 1$ such that every element of $M$ factors into a product of $k$ or fewer i\-dem\-po\-tents, with the understanding that $\inf \emptyset := \infty$ (note that, since the identity $1_M$ is an idempotent, an element of $M$ factors as a product of $m$ idempotents if and only if it factors as a product of $n$ idempotents for every $n \ge m$).

\begin{claim}\label{claimA}
$\mathcal E(H) = \prod_{i \in I} \mathcal E(H_i)$ if and only if $\Delta(\mathcal E(H_i)) = \infty$ for finitely many $i \in I$.
\end{claim}

\begin{proof}
To ease notation, set $\Delta_i := \Delta(\mathcal E(H_i))$ for each $i \in I$ and then $J := \{j \in I \colon \Delta_j = \infty\}$. We need to check that $\mathcal E(H) = \prod_{i \in I} \mathcal E(H_i)$ if and only if $J$ is a finite set.

Suppose first that $J$ is finite and let $f$ be a function in $\prod_{i \in I} \mathcal E(H_i)$. Then $\Delta := 1 + \sup_{i \in I \setminus J} \Delta_i$ is a positive integer (with $\sup \emptyset := 0$) and, for each $j \in J$, $f(j)$ factors into a product of $k_j$ idempotents of $H_j$ for some $k_j \in \mathbb{N}^+$. Hence, there exists an integer $k \ge 1$ such that, for every $i \in I$, $f(i)$ is a product of $k$ idempotents $e_i^{(1)} \ldots, e_i^{(k)} \in H_i$ (some of which are possibly equal to the identity $1_{H_i}$); for instance, we may take $k$ to be the maximum of the (non-empty, finite) set $\{\Delta\} \cup \{k_j \colon j \in J\} \subseteq \mathbb{N}^+$. As a result, $f$ is a product of $k$ idempotents $e_1, \ldots, e_k \in H$, where, for each $j \in \llb 1, k \rrb$, $e_j$ is the function $I \to S \colon i \mapsto e_i^{(j)}$. Therefore, we gather from Eq.~\eqref{equ:2021-06-17-17:44} that $\prod_{i \in I} \mathcal E(H_i) = \mathcal E(H)$.

Assume, on the other hand, that $J$ is infinite. There is then a sequence $i_1, i_2, \ldots$ of pairwise distinct indices in $I$ with $\Delta_{i_j} \ge j$ for all $j \in \mathbb{N}^+$. Consequently, we have an $S$-valued sequence $x_1, x_2, \ldots$ such that, for each $j \in \mathbb{N}^+$, $x_j$ is an element of $\mathcal E(H_{i_j})$ whose shortest idempotent-factorization in $H_{i_j}$ has at least $j$ factors. Accordingly, we set $I' := \{i_1, i_2, \ldots\}$ and define $g$ to be the function 
\[
I \to S \colon i \mapsto 
\left\{
\begin{array}{ll}
\! x_i & \text{if } i \in I'  \\
\! 1_{H_i}   & \text{if } i \in I \setminus I'
\end{array}
\right..
\]
 
By construction, $g \in \prod_{i \in I} \mathcal E(H_i)$. However, $g \notin \mathcal E(H)$ and hence $\mathcal E(H) \neq \prod_{i \in I} \mathcal E(H_i)$. To see this, suppose to the contrary that $g$ factors as a product of $n$ idempotents of $H$ for some $n \in \mathbb{N}^+$. Then $g(i)$ is a product of $n$ idempotents of $H_i$ for every $i \in I$, which is impossible because we have $g(i_{n+1}) = x_{n+1}$ and, again by construction, $x_{n+1}$ does not factor into less than $n+1$ idempotents of $H_{i_{n+1}}$.
\end{proof}

We note for later reference that
the minimum length $\ell(f)$ of an idempotent factorization of a function $f \in \mathcal E(H)$ is bounded above by $ \max_{i \in I} \ell_i(f)$, where we denote by $\ell_i(f)$ the minimum length of an i\-dem\-po\-tent factorization of $f(i)$ in $\mathcal E(H_i)$. In fact, $f = e_1 \cdots e_n$ for some idempotents $e_1, \ldots, e_n \in H$, with the result that $\ell_i(f) \leq n$ for every $i \in I$. Therefore, the (non-empty) set $\{\ell_i(f) \colon i \in I\} \subseteq \mathbb{N}$ has a maximum, say $m$; and similarly as in the proof of Claim \ref{claimA}, we see that, for each $i \in I$, $f(i)$ factors as a product of $m$ idempotents of $H_i$ (some of which may be equal to $1_{H_i}$). This, in turn, shows that $f$ can be written as a product of $m$ idempotents of $H$. That is to say, $\ell(f) \le m$ (as wished). 

\begin{claim}\label{claimB}
If $H_j^\# \subsetneq H_j$ and $H_k^\# \ne \{1_{H_k}\}$ for some $j, k \in I$ with $j \ne k$, then $\prod_{i \in I} H_i^\# \subsetneq H^\#$.
\end{claim}

\begin{proof}
%Assume first that 
Assume $H_j^\# \subsetneq H_j$ and $H_k^\# \ne \{1_{H_k}\}$ for some $j, k \in I$ with $j \ne k$, so that we can pick an element $a \in H_j \setminus H_j^\#$ and a singular element $b \in H_k^\#$. There then exist $x, y, z, w \in H_k$ with $x \ne y$ and $z \ne w$ such that $xb = yb$ and $bz = bw$.
%(the multiplications are carried out in $H_k$). 
Therefore, the function $f := \delta_j^{(a)} \delta_k^{(b)}$ is in $H^\#$, because $\delta_k^{(x)} f = \delta_k^{(y)} f$ and $f \delta_k^{(z)} = f \delta_k^{(w)}$ with $\delta_k^{(x)} \ne \delta_k^{(y)}$ and $\delta_k^{(z)} \ne \delta_k^{(w)}$. Yet, $f$ is not  in $\prod_{i \in I} H_i^\#$, by the fact that $f(j) = a \notin H_j^\#$. Consequently, we gather from Eq.~\eqref{equ:2021-06-17-17:44} that $H^\#$ is properly contained in $\prod_{i \in I} H_i^\#$.
\end{proof}

If $I$ is a finite set or $\Delta(\mathcal E(H_i)) < \infty$ for each $i \in I$, we obtain from Eq.~\eqref{equ:2021-06-17-17:44} and Claim \ref{claimA} that $\mathcal E(H) = \allowbreak \prod_{i \in H_i} H_i^\#$ if and only if $H_i^\#$ is idempotent-generated for every $i \in I$. If, on the other hand, each of the monoids $H_i$ is the mul\-ti\-pli\-ca\-tive monoid of a (non-zero) ring $R_i$, then $H_i^\#$ is a non-trivial monoid for all $i \in I$ and hence we derive from Claim \ref{claimB} that $H^\# = \prod_{i \in H_i} H_i^\#$ only if, for all $i \in I$, any element of $R_i$ other than the identity is a zero divisor. Rings with this property are known as ``$\mathcal O$-rings'' and were first considered by P.\,M.~Cohn in \cite{Cohn58} (though only in the commutative case). While it was proved by H.\,G.~Moore et al.~\cite{Mo-Pi-Ya68} that any right (or left) artinian $\mathcal O$-ring is commutative, it appears to be open \cite{Hen} whether non-commutative $\mathcal O$-rings do actually exist (see also \url{https://mathoverflow.net/questions/395787/}).
\end{remark}

\section{Idempotent factorizations in arbitrary rings}
\label{sec:4}
We are going to apply the abstract results of Sect.~\ref{sec:3} to the study of idempotent factorizations in the multiplicative monoid of certain rings (it may help to review Sects.~\ref{sec:monoids} and \ref{sec:2.3} before reading further). 
\vskip 0.2cm
\begin{quotation}
\centering
\emph{Throughout this section \textup{(}unless noted otherwise\textup{)}, we let $R$ be an arbitrary ring.}
\end{quotation}
\vskip 0.2cm
Our goal is to prove a series of results that will be used in Sect.~\ref{subsec:characterizations} to characterize or partially characterize the $\rfix_R$-irreducibles and $\rfix_R$-quarks of the monoid of zero divisors of $R$ when $R$ is a right Rickart ring (Proposition \ref{prop:4.6}, Theorem \ref{thm:rFix-quarks-in-right-rickart}, and Remark \ref{rem:left-sing-implies-sing}), a von Neumann regular ring satisfying the comparability axiom (Theorem \ref{thm:5.8}), or the endomorphism ring of a free module of finite rank over a commutative discrete valuation domain (Proposition \ref{dge1}). 
We start with a lemma that is perhaps of independent interest.

\begin{lemma}\label{lem:2021-05-05-11:04}
If $\rann(a)$ is a direct summand of $R_R$ for some $a \in R$, then there is an idempotent $e \in R$ such that $
\rann(a) = eR$ and $\rfix(a) \subseteq (1_R-e)R$.
\end{lemma}

\begin{proof}
Let $\rann(a)$ be a direct summand of $R_R$. By Proposition \ref{prop:2.2}\ref{prop:2.2(i)}, there exists an i\-dem\-po\-tent $f \in R$ such that $\rann(a) = fR$. Set $e := f-fa$. Since $fe = f^2(1_R-a) = e$ and $ef = f^2-faf = f$ (note that $af = 0_R)$, we have $eR=fR$. Moreover, we check that $e^2= fe - faf(1_R-a) = e$, i.e., $e$ is an i\-dem\-po\-tent of $R$. Lastly, if $x \in R$ and $(1_R-a)x = 0_R$, then $(1_R-e)x = x - f(1_R-a)x = x - 0_R = x$, which shows that $\rfix(a)$ is contained in $(1_R-e)R$.
\end{proof}

We continue with a couple of technical lemmas whose role will be clarified by Proposition \ref{prop:4.6}.

\begin{lemma}\label{lem:2021-04-29-15:28}
Assume there exist $a \in R$ and right ideals $\mathfrak i, \mathfrak j \subseteq R$ such that $R = \mathfrak i \oplus \mathfrak j$ and $\mathfrak i \subseteq \rann(a)$, and set $\mathfrak{p} := \rann(a^2) \cap \mathfrak{j}$ and $\mathfrak{q} := \{x+ax \colon x \in \mathfrak p\}$. The following hold:
\begin{enumerate}[label=\textup{(\roman{*})}]
\item\label{lem:2021-04-29-15:28(i)} $\rann(a) \cap \mathfrak{j} = \{0_R\}$ if and only if $\mathfrak i = \rann(a)$.

\item\label{lem:2021-04-29-15:28(ii)} $\mathfrak{p}$ and $\mathfrak{q}$ are right ideals of $R$ with $\mathfrak{q} \cap \rfix(a) = \{0_R\}$. 

\item\label{lem:2021-04-29-15:28(iii)} $\mathfrak{p} = \{0_R\}$ if and only if $\mathfrak{q} = \{0_R\}$, if and only if $\mathfrak{i} = \rann(a) = \rann(a^2)$.
\end{enumerate}
\end{lemma}

\begin{proof}
\ref{lem:2021-04-29-15:28(i)}: Since $\mathfrak i$ and $\mathfrak j$ are ideals of $R$ with $\mathfrak{i} \oplus \mathfrak{j} = R$, it is clear that $\mathfrak{i} = \rann(a)$ implies $\rann(a) \cap \mathfrak{j} = \{0_R\}$. As for the converse, assume $\mathfrak{i} \subsetneq \rann(a)$ and, accordingly, let $x$ be an element in $\rann(a)$ but not in $\mathfrak{i}$ (recall that, by hypothesis, $\rann(a)$ contains $\mathfrak{i}$). Then $x = y+z$ for some $y \in \mathfrak i$ and $z \in \mathfrak j \setminus \{0_R\}$, with the result that
$
0_R = \allowbreak ax = ay + az = az$.
This proves $0_R \ne z \in \rann(a) \cap \mathfrak{j}$, and we are done.

\vskip 0.05cm

\ref{lem:2021-04-29-15:28(ii)}: The intersection of two right ideals is still a right ideal; therefore, $\mathfrak p$ is a right ideal, because so are $\rann(a^2)$ and $\mathfrak j$. On the other hand, if $x, y \in \mathfrak p$ and $b, c \in R$, then 
\[
(x+ax)b + (y+ay)c = (xb+yc) + a(xb+yc) \in \mathfrak q, 
\]
by the fact that $\mathfrak p$ is a right ideal and hence $xb+yc \in \mathfrak p$. So, $\mathfrak q$ is a right ideal too.

To complete the argument, let $x \in \mathfrak q \cap \rfix(a)$. Then $ax = x$ and $x = y + ay$ for some $y \in \mathfrak p \subseteq \allowbreak \rann(a^2)$, so that
\[
y + ay = a(y + ay) = ay + a^2y = ay + 0_R = ay.
\]
It follows that $y = 0_R$ and hence $x = y + ay = 0_R$. As a result, $\mathfrak q$ and $\rfix(a)$ intersect trivially.

\vskip 0.05cm

\ref{lem:2021-04-29-15:28(iii)}: Suppose first that $\rann(a^2) \setminus \rann(a)$ is a non-empty set and let $x$ be one of its elements. Since $R = \mathfrak i \oplus \mathfrak j$ and $\mathfrak i \subseteq \rann(a)$, we have $x = y+z$ for some $y \in \rann(a)$ and $z \in \mathfrak j \setminus \{0_R\}$. Consequently,
\[
0_R = a^2x = a^2(y+z) = a(ay + az)= a(0_R + az) = a^2 z,
\]
from which we gather that $z$ is a non-zero element of $\mathfrak p = \mathfrak j \cap \rann(a^2)$.

Since $\rann(a) \subseteq \rann(a^2)$, we obtain from the above that, if $\mathfrak p = \{0_R\}$, then $\rann(a) = \rann(a^2)$ and hence $\{0_R\} = \mathfrak p = \mathfrak j \cap \rann(a)$. On the other hand, we have from item \ref{lem:2021-04-29-15:28(i)} that $\mathfrak j \cap \rann(a) \allowbreak = \{0_R\}$ if and only if $\mathfrak{i} = \rann(a)$. So putting it all together, we conclude that if $\mathfrak p = \{0_R\}$ then $\mathfrak{i} = \rann(a) \allowbreak = \rann(a^2)$; the converse is trivial, because $\mathfrak{i} \cap \mathfrak{j} = \{0_R\}$.

It remains to show that $\mathfrak{q} = \{0_R\}$ if and only if $\mathfrak{p} = \{0_R\}$. Since the ``if'' part is obvious, it suffices to check that if $\mathfrak{p} \ne \{0_R\}$ then $\mathfrak{q} \neq \{0_R\}$. For, assume $\mathfrak p$ contains an element $x \ne 0_R$. Then either $ax = 0_R$ and hence $x = x + ax \in \mathfrak q \setminus \{0_R\}$; or $0_R \ne ax = ax + \allowbreak a^2x = \allowbreak a(x+ax)$ and hence $x + ax \in \mathfrak{q} \setminus \{0_R\}$.
\end{proof}

\begin{lemma}\label{lem:2021-04-29-12:26}
Let $e \in R$ be an idempotent in the right annihilator $\rann(a)$ of an element $a \in R$, and set $\mathfrak{p} := \allowbreak \rann(a^2) \allowbreak \cap \allowbreak (1_R-e)R$ and $\mathfrak{q} := \{x+ax \colon x \in \mathfrak p\}$. 
The following hold:
	
\begin{enumerate}[label=\textup{(\roman{*})}] 
\item\label{lem:2021-04-29-12:26(i)} $a = bc$ and $c^2 = c$, where $b := e + (1_R - e)a$ and $c := 1_R-e + ea$.
		
\item\label{lem:2021-04-03-12:06(ii)} 
$\rann(a) \cap (1_R-e)R \subseteq \rann(b)$ and $eR \oplus \rfix(a) \subseteq \rfix(b)$. 
		
\item\label{lem:2021-04-03-12:06(iii)} $eR = \rann(c)$ and $\rfix(a) \oplus (\rann(a) \cap (1_R-e)R) \subseteq \rfix(c)$.

\item\label{lem:2021-05-01-11:43(iv)} If $eR = \rann(a)$, then $\mathfrak{p} \subseteq \rann(b)$ and $\mathfrak{q} \oplus \rfix(a) \subseteq \rfix(c)$.
\end{enumerate}
\end{lemma}

\begin{proof}
\ref{lem:2021-04-29-12:26(i)}: Since $e$ is an i\-dem\-po\-tent element of $R$ from the right annihilator of $a$, we have 
\begin{equation}\label{equ:2021-05-01-10:35}
e(1_R - e) = (1_R - e)e = ae = 0_R
\quad\text{and}\quad
a(1_R - e) = a - ae = a \in R(1_R-e).
\end{equation}
Consequently, it is clear that $ea \in eR(1_R-e)$ and we get from Proposition \ref{prop:2.2}\ref{prop:2.2(v)} (applied with $z = ea$) that $c$ is an idempotent of $R$. Moreover, a simple calculation shows that
\[
bc = e(1_R - e) + e^2 a + (1_R - e)\, a (1_R - e) + (1_R - e)\, aea = 0_R + ea + (1_R - e) a + 0_R = a.
\]

\ref{lem:2021-04-03-12:06(ii)} and \ref{lem:2021-04-03-12:06(iii)}: It is easily checked that, for every $x \in \rfix(a)$,
\begin{equation}\label{equ:2021-05-01-10:38}
bx = ex + (1_R-e)ax = ex + (1_R - e)x = x
\quad\text{and}\quad
cx = (1_R - e)x + eax = x;
\end{equation}
while it follows from Eq.~\eqref{equ:2021-05-01-10:35} (and $e$ being i\-dem\-po\-tent) that
\begin{equation}\label{equ:2021-05-01-10:39}
be = e^2 + (1_R-e) ae = e + 0_R = e
\quad \text{and}\quad
ce = (1_R-e) e + eae = 0_R + 0_R = 0_R.
\end{equation}
Likewise, if $x \in \rann(a) \cap (1_R-e)R$, then $x = (1_R - e)y$ for some $y \in R$ and hence
\begin{equation}\label{equ:2021-05-01-10:40}
bx = e(1_R-e)y + (1_R-e) ax = 
0_R + 0_R = 0_R
\quad\text{and}\quad
cx = x - e(1_R-e)y + eax = x.
\end{equation}
So, recalling that $e \in \rann(a)$,
we conclude from Eqs.~\eqref{equ:2021-06-01-14:47} and \eqref{equ:2021-05-01-10:38}--\eqref{equ:2021-05-01-10:40} that
\[
\rann(a) \cap (1_R-e)R \subseteq \rann(b)
\quad\text{and}\quad
eR \oplus \rfix(a) \subseteq \rfix(b)
\]
and
\[
eR \subseteq \rann(c)
\quad\text{and}\quad
(\rann(a) \cap (1_R-e)R) \oplus \rfix(a) \subseteq \rfix(c),
\]
where, in particular, we have used that if $\mathfrak h$, $\mathfrak i$, and $\mathfrak j$ are ideals of $R$ with $\mathfrak i \cup \mathfrak j \subseteq \mathfrak h$, then $\mathfrak i + \mathfrak j \subseteq \mathfrak h$.

It remains to see that $\rann(c) \subseteq eR$. For, let $x \in \rann(c)$. Since $R = eR \oplus (1_R-e)R$, we can write $x = y + z$, where $y \in eR$ and $z \in (1_R-e)R$. Then $y = eu$ and $z = (1_R-e)v$ for some $u, v \in R$, which implies, by Eq.~\eqref{equ:2021-05-01-10:35}, that $(1_R-e)y = ay = 0_R$ and hence $cy = 0_R$. As a result, 
\[
0_R = cx = cz = (1_R-e)z + eaz = (1_R-e)^2v + eaz = (1_R-e)v + eaz = z + eaz.
\]
This, however, is only possible if $z = eaz = 0_R$, because $eR \cap (1_R-e)R = \{0_R\}$ and, on the other hand, $z \in (1_R-e)R$ and $eaz \in eR$. Therefore, we obtain $x = y+z = y \in eR$ (as wished).

\vskip 0.05cm

\ref{lem:2021-05-01-11:43(iv)}: Suppose $eR = \rann(a)$, and pick $x \in \mathfrak p$ and $y \in \mathfrak q$. By the very definition of $\mathfrak p$ and $\mathfrak q$ (and, more specifically, by the fact that $\mathfrak p$ is contained in $(1_R-e)R$), there then exist $\bar{x}, \bar{y} \in R$ with $x = (1_R-e)\bar{x}$ and $y = z+az$, where $z := (1_R-e)\bar{y} \in \mathfrak p$.
Moreover, $a^2x = a^2z = 0_R$ (because $\mathfrak p$ is also contained in the right annihilator of $a^2$) and hence $ax, az \in \rann(a) = eR$, implying that $ax = eu$ and $az = ev$ for some $u, v \in R$. Consequently, a simple calculation shows that
\[
bx = ex + (1_R - e)ax = e(1_R-e)\bar{x} + (1_R - e)eu = 0_R + 0_R = 0_R;
\]
and on the other hand, we have 
\[
(1_R-e)z = (1_R-e)^2 \bar{y} = (1_R-e)\bar{y} = z
\quad\text{and}\quad
eaz = e^2v = ev = az,
\]
which, in turn, yields  
\[
cy = (1_R-e)z + (1_R-e)az + eaz + ea^2z = z + 0_R + az + 0_R = y.
\]
We thus see that $\mathfrak p \subseteq \rann(b)$ and $\mathfrak q \subseteq \rfix(c)$; and by Lemma \ref{lem:2021-04-29-15:28}\ref{lem:2021-04-29-15:28(ii)} and the ob\-ser\-va\-tion already made on the last line of the proof of items \ref{lem:2021-04-03-12:06(ii)} and \ref{lem:2021-04-03-12:06(iii)}, we conclude that $\mathfrak{q} \oplus \rfix(a) \subseteq \rfix(c)$.
\end{proof}

The lemmas we have hitherto proved make it already possible to put non-trivial constraints on the $\rfix_R$-irreducibles and $\rfix_R$-quarks of the monoid of left zero divisors of $R$. Sharper conditions will be obtained in Sect.~\ref{subsec:characterizations} (in particular, see Theorem \ref{thm:rFix-quarks-in-right-rickart}).

\begin{proposition}
\label{prop:4.6}
Let $H$ be the monoid of left zero divisors of a ring $R$ and $a$ be an element of $R$ such that $\rann(a)$ is a direct summand of $R_R$. The following hold:
\begin{enumerate}[label=\textup{(\roman{*})}]
\item\label{prop:4.6(i)} If $a$ is an $\rfix_R$-irreducible of $H$,
then $\rann(a)$ is an indecomposable submodule of $R_R$ with the additional property that $\rann(a) = \rann(a^2)$ and $\rann(a) \cap aR = \{0_R\}$.
\item\label{prop:4.6(ii)} If $a$ is an $\rfix_R$-quark of $H$, then $a$ is a coprimitive idempotent of $R$.
\end{enumerate}

\end{proposition}

\begin{proof}
\ref{prop:4.6(i)}: Let $a$ be an $\rfix_R$-irreducible of $H$. Then $a$ is a left zero divisor of $R$. Since $\rann(a)$ is a direct summand of $R_R$ (by hypothesis), we get from Proposition \ref{prop:2.2}\ref{prop:2.2(i)} that
there are an idempotent $e \in \allowbreak R$ and a right ideal $\mathfrak l \subseteq R$ such that
\begin{equation}\label{equ:2021-06-03-16:40}
R = \rann(a) \oplus \mathfrak l
\quad\text{and}\quad
\rann(a) = eR \ne \{0_R\}.
\end{equation}
We divide the rest of the proof into two parts. First, we prove $\rann(a)$ is an indecomposable submodule of $R_R$. Next, we check $\rann(a) = \rann(a^2)$. This will suffice to prove the statement, for it implies at once that $\rann(a)$ and $aR$ intersect trivially. (If $y \in \rann(a) \cap aR$, then $ay = 0_R$ and $y = az$ for some $z \in \allowbreak R$, so that $0_R = ay = a^2z$. It follows that $z \in \rann(a^2) = \rann(a)$ and hence $y = az = 0_R$.)

\vskip 0.05cm 

\textsc{Part 1:} 
Assume to the contrary that $\rann(a)$ is not an indecomposable submodule of $R_R$. Since $\rann(a) \ne \{0_R\}$, it follows that $\rann(a) = \mathfrak g \oplus \mathfrak h$ for some non-zero right ideals $\mathfrak g, \mathfrak h \subseteq R$, which, by Eq.~\eqref{equ:2021-06-03-16:40}, yields $R = \mathfrak g \oplus \mathfrak h \oplus \mathfrak l$. In view of Lemma \ref{lem:2021-04-29-15:28}\ref{lem:2021-04-29-15:28(i)} (applied with $\mathfrak i = \mathfrak g$ and $\mathfrak j = \mathfrak h \oplus \mathfrak l$) and Prop\-o\-si\-tion \ref{prop:2.2}\ref{prop:2.2(i)}, there then exists an idempotent $f \in R$ such that 
\[
\{0_R\} \ne \mathfrak g = fR \subsetneq \rann(a)
\quad\text{and}\quad
\mathfrak n := \rann(a) \cap (1_R - f)R \ne \{0_R\}. 
\]
Consequently, we conclude from Lemma \ref{lem:2021-04-29-12:26} that $a = bc$ for some $b, c \in R$ with 
\[
\{0_R\} \ne \mathfrak n \subseteq \rann(b)
\quad\text{and}\quad
\{0_R\} \ne fR = \rann(c)
\]
and
\[
\rfix(a) \subsetneq fR \oplus \rfix(a) \subseteq \rfix(b)
\quad\text{and}\quad
\rfix(a) \subsetneq \rfix(a) \oplus \mathfrak n \subseteq \rfix(c).
\]
This however means that $b$ and $c$ are left zero divisors of $R$ each of which is (strictly) smaller than $a$ with respect to the $\rfix_R$-preorder on $H$. Therefore, $a$ is not an $\rfix_R$-irreducible of $H$ (which is absurd), because the only $\rfix_R$-unit of $H$ is the identity $1_H$ (which is not a left zero divisor of $R$). 

\vskip 0.05cm 

\textsc{Part 2:} Suppose for the sake of a contradiction that $\rann(a) \ne \rann(a^2)$ and hence $\rann(a) \subsetneq \rann(a^2)$. Then we get from items \ref{lem:2021-04-29-15:28(ii)} and \ref{lem:2021-04-29-15:28(iii)} of Lemma \ref{lem:2021-04-29-15:28} (applied with $\mathfrak i = eR$ and $\mathfrak j = (1_R - e)R$) that 
\[
\mathfrak{p} := \rann(a^2) \cap (1_R - e)R
\quad\text{and}\quad
\mathfrak{q} := \{x+ax \colon x \in \mathfrak p\}
\]
are non-zero right ideals of $R$. So, we are guaranteed by Lemma \ref{lem:2021-04-29-12:26} that $a = bc$ for some $b, c \in R$ with 
\[
\{0_R\} \ne \mathfrak p \subseteq \rann(b)
\quad\text{and}\quad
\{0_R\} \ne \rann(a) = \rann(c)
\]
and
\[
\rfix(a) \subsetneq \rann(a) \oplus \rfix(a) \subseteq \rfix(b)
\quad\text{and}\quad
\rfix(a) \subsetneq \mathfrak q \oplus \rfix(a) \subseteq \rfix(c).
\]
But similarly as in \textsc{Part 1}, this entails that $a$ is not an $\rfix_R$-irreducible of $H$ (which is absurd).

\vskip 0.05cm

\ref{prop:4.6(ii)}: 
Assume by way of contradiction that $a$ is an $\rfix_R$-quark of $H$ but not a coprimitive i\-dem\-po\-tent of $R$. Then $a$ is an $\rfix_R$-irreducible of $H$ (see the comments after Definition~\ref{def:2021-06-01-23:10}); and by item \ref{prop:4.6(i)}, this is only possible if $\rann(a)$ is an indecomposable submodule of $R_R$. On the other hand, we have from Eq.~\eqref{equ:2021-06-03-16:40} and Lemma \ref{lem:2021-05-05-11:04} that $\rann(a) = e_0R \ne \{0_R\}$ and $\rfix(a) \subsetneq (1_R - e_0)R$ for a certain idempotent $e_0 \in R$. In particular, 
$\rfix(a)$ is not equal to $(1_R - e_0)R$, or else we would gather from the above and Proposition \ref{prop:caratterizzazione-idemps} that $a$ is a coprimitive idempotent (which is absurd). So, $e_0$ is a non-zero element of $R$ and we conclude from Eq.~\eqref{equ:ann&fix-of-idempotent} that $1_R-e_0$ is a zero divisor of $R$ with $\rfix(a) \subsetneq \allowbreak (1_R - e_0)R = \allowbreak \rfix(1_R-e_0) \subsetneq R$, contradicting that $a$ is an $\rfix_R$-quark of $H$.
\end{proof}

There is more that we can gain from the last proof, but we will come back to this point in Remark \ref{rem:5.3}, after proving in Theorem \ref{thm:rFix-quarks-in-right-rickart} that, for a right Rickart ring, the conclusions of Proposition \ref{prop:4.6}\ref{prop:4.6(ii)} can be turned into an ``if and only if'' statement.
\section{Idempotent factorizations in right Rickart rings} 
\label{subsec:characterizations}

We recall from Sect.~\ref{sec:2.3} that a ring $R$ is a right Rickart ring if $\rann(a)$ is a direct summand of the regular right module $R_R$ for every $a \in R$. Some remarks are in order before proceeding.
\begin{remarks}\label{rem:right-rickart-rings}
\begin{enumerate*}[label=\textup{(\arabic{*})}]
\item\label{rem:right-rickart-rings(1)} First considered by S.~Maeda \cite{Ma60b} and A.~Hattori \cite{Hat60} and known also as ``principal projective rings'' or ``p.p.~rings'' (see \cite[Sect.~1]{Lee-TaRi-Ro2010} for further historical details), 
right Rickart rings form a wide class of rings, including von Neumann regular rings and, more generally, right semi-hereditary rings,
see \cite[Corollary 1.2(d)]{Go79} and the comments on the bottom of \cite[p.~261]{Lam99}. In addition, we gather from \cite[Proposition 3.2]{Lee-TaRi-Ro2010} that the same class also includes the endo\-mor\-phism ring of a \evid{Rickart} (\evid{right}) \evid{module}, by which we mean a (right) module $M$ where the kernel of each endo\-mor\-phism is a direct summand of $M$, see \cite[Definition 2.2]{Lee-TaRi-Ro2010}. Most notably, we have from \cite[Theorem 3.6]{Lee-TaRi-Ro2012} that every finitely generated, projective (or free) right module $M$ over a right semi-hereditary ring is a Rickart module.\\[0.05cm]

\indent\item\label{rem:right-rickart-rings(2)} If $R$ is a right Rickart ring and $\rfix(a) \subseteq \rfix(b)$ for some $a, b \in R$, then we get from Eq.~\eqref{equ:r.fix-in-a-ring} that each of $\rfix(a)$ and $\rfix(b)$ is a principal right ideal of $R$ generated by an idempotent, which, by Proposition \ref{prop:2.2}\ref{prop:2.2(iv)}, implies the existence of a right ideal $\mathfrak j \subseteq R$ such that $\rfix(v) = \mathfrak j \oplus \rfix(u)$.\\[0.05cm]
\indent{}\item\label{rem:right-rickart-rings(5)}
Let $R$ be a ring. The ring $T_n(R)$ of $n$-by-$n$ upper triangular matrices over $R$ is finitely generated as a left (resp., right) $R$-module; it thus follows from \cite[Proposition (1.21)]{Lam01} (and its right analogue) and the Hopkins-Levitzki theorem that, if $R$ is an artinian ring, then $T_n(R)$ is artinian and hence noetherian. On the other hand, we gather from \cite[Example (2.36)]{Lam99} that, if $R$ is a skew field, then $T_n(R)$ is a hereditary ring and hence, by Remark \ref{rem:right-rickart-rings}\ref{rem:right-rickart-rings(1)}, a right Rickart ring. Yet, $T_n(R)$ is not a von Neumann regular ring for $n \ge 2$. For, if $A \in T_n(R)$ is a matrix whose $(1,n)$-entry is non-zero and all of whose other entries are zero, then $AXA = 0_{T_n(R)} \ne A$ for every $X \in T_n(R)$.
\end{enumerate*}
\end{remarks}

The role of right Rickart rings in our approach to the study of idempotent factorizations should be already clear from the centrality of right annihilators in the results of Sect.~\ref{sec:2.1}, and it will become even clearer (we hope) with the next theorem.

\begin{theorem}\label{thm:rFix-quarks-in-right-rickart}
Let $H$ be the monoid of left zero divisors of a right Rickart ring $R$. An element $a \in R$ is an $\rfix_R$-quark of $H$ if and only if $a$ is a coprimitive idempotent. 
\end{theorem}

\begin{proof}
%\ref{thm:rFix-quarks-in-right-rickart(i)} 
The ``only if'' part is immediate from Proposition \ref{prop:4.6}\ref{prop:4.6(ii)}.
For the ``if'' part, let $a$ be a coprimitive idempotent of $R$ and suppose by way of contradiction that $a$ is not an $\rfix$-quark of $H$ (note that every idempotent of $R$ is an element of $H$). There then exists $b \in H$ with $b \ne 1_R$ such that $\rfix(a) \subsetneq \rfix(b)$, and so we see from Remark \ref{rem:right-rickart-rings}\ref{rem:right-rickart-rings(2)} that $R = \mathfrak i \oplus \rfix(b)$ and $\rfix(b) = \mathfrak j \oplus \rfix(a)$ for some non-zero right ideals $\mathfrak i, \mathfrak j \subseteq R$. In particular, $\mathfrak i$ is non-zero because $b$ is a left zero divisor of $R$ and hence $\rfix(b)$ is properly contained in $R$.
It thus follows by Proposition \ref{prop:caratterizzazione-idemps} that
\begin{equation*}\label{equ:2021-06-16-19:00}
\rann(a) \oplus \rfix(a) = R = \mathfrak i \oplus \rfix(b) = \mathfrak i \oplus \mathfrak j \oplus \rfix(a).
\end{equation*}
But this implies by \cite[Exercise 5.4(1), p.~76]{An-Fu92} that $\rann(a)$ is isomorphic, as a right $R$-module, to $\mathfrak i \oplus \mathfrak j$, which is impossible because $\rann(a)$ is an indecomposable submodule of $R_R$ (by the hypothesis that $a$ is a coprimitive idempotent) and, on the other hand, $\mathfrak i$ and $\mathfrak j$ are both non-zero right ideals of $R$.
\end{proof}

\begin{remark}\label{rem:5.3}
Let $H$ be the monoid of left zero divisors of a right Rickart ring $R$ and $\preceq$ be the restriction to $H$ of the $\rfix$-preorder on (the multiplicative monoid of) $R$, and pick $a \in H$. Since $R$ is right Rickart, $\rann(a)$ is a direct summand of $R_R$. So, we obtain from Proposition \ref{prop:4.6}\ref{prop:4.6(i)} that $a$ is a $\preceq$-irreducible of $H$ only if $\rann(a)$ is an indecomposable submodule of $R_R$ with $\rann(a) = \rann(a^2)$. If, on the other hand, $\rann(a) = \mathfrak g \oplus \mathfrak h$ for some non-zero right ideals $\mathfrak g, \mathfrak h \subseteq R$ such that $\mathfrak g$ is an in\-de\-com\-pos\-a\-ble $R$-module (which, by \cite[Corollary (6.7)(1)]{Lam99}, is guaranteed when $R$ is a right noetherian ring) or $\rann(a)$ is an indecomposable $R$-module with $\rann(a) \ne \rann(a^2)$, then the proof of the same Proposition \ref{prop:4.6}\ref{prop:4.6(i)} shows that $a$ factors as the product of two left zero divisors $b, c \in R$ with $b \prec a$ and $c \prec a$ such that $\rann(c)$ is an indecomposable $R$-module; that is, $c$ is a coprimitive idempotent of $R$ and hence a $\preceq$-quark (Theorem \ref{thm:rFix-quarks-in-right-rickart}), which gives in turn that $\hgt(b) + \hgt(c) \leq \hgt(a)$, where $\hgt(\cdot)$ is the $\preceq$-height (function) of $H$.
\end{remark}

The next remarks show that, in certain classes of rings, the monoid of left zero divisors coincides with the monoid of zero divisors, which makes Proposition \ref{prop:4.6} and Theorem \ref{thm:rFix-quarks-in-right-rickart} smoother.

\begin{remarks}\label{rem:left-sing-implies-sing}
\begin{enumerate*}[label=\textup{(\arabic{*})}]
\item\label{rem:left-sing-implies-sing(1)}
A ring $R$ is \evid{strongly $\pi$-regular} if there is no $a \in R$ with $a^{k+1} R \subsetneq a^k R$ for all $k \in \mathbb{N}^+$ (e.g., every left or right artinian ring is strongly $\pi$-regular). By the comments under the solution to Exercise 4.17 in \cite{Lam03}, the elements of a strongly $\pi$-regular ring are either units or zero divisors. In particular, this is true of the ring of $n$-by-$n$ upper triangular matrices with entries in a skew field (Remark \ref{rem:right-rickart-rings}\ref{rem:right-rickart-rings(5)}). \\[0.05cm]

\indent{}\item\label{rem:left-sing-implies-sing(2)}
Every non-unit element of a Dedekind-finite, von Neumann regular monoid $H$ is singular. In fact, let $x \in H$ be a non-unit. Since $H$ is von Neumann regular, we have $x = xyx$ for some $y \in H$. Therefore, if $x$ is left- or right-cancellative (namely, $x$ is not a zero divisor), then $xy = 1_H$ or $yx = 1_H$, which, by the Dedekind-finiteness of $H$, implies that $x$ is a unit (which is absurd).\\[0.05cm]
\indent{}\item\label{rem:left-sing-implies-sing(3)} A (right) module $M$ is said to be \evid{hopfian} if every surjective endomorphism of $M$ is an automorphism. Consequently, an element $f$ in the endomorphism ring of a hopfian module is a zero divisor if and only if $f$ is a \emph{left} zero divisor, if and only if $f$ is a non-injective map.
Most notably, we have from (the right analogue of) \cite[Exercise 1.12(1)]{Lam03} that every noetherian module $M$ is hopfian, and by (the right analogue of) \cite[Proposition (1.21)]{Lam01}, this is especially the case when $M$ is a finitely generated module over a right noetherian ring.
\end{enumerate*}
\end{remarks}

We continue by proving that, in a right Rickart ring $R$ such that the regular right module $R_R$ has finite uniform dimension, the $\rfix_R$-preorder is strongly artinian.

\begin{proposition}\label{prop:rickart-strongly-fix-artinian}
Let $R$ be a right Rickart ring with finite uniform dimension. Then the $\rfix_R$-height of an element $a \in R$ relative to the mul\-ti\-pli\-ca\-tive monoid of $R$ is finite and bounded above by $\udim(R_R) - \rdim_R(\rfix(a))$. In particular, every multiplicative sub\-monoid of $R$ is strongly $\rfix_R$-artinian.
\end{proposition}

\begin{proof}
Set $N := \udim(R_R)$ and fix $a \in R$. It is enough to check that $\alpha + h \leq N$, where $\alpha$ is the right uniform dimension of $\rfix(a)$ and $h$ is the $\rfix_R$-height of $a$ relative to the monoid $H := (R, \cdot\,)$. Since $N$ is finite (by hypothesis) and the only $\rfix_R$-unit of $H$ is the identity $1_R$, this will in fact prove that $H$ is a strong\-ly $\rfix_R$-artinian monoid and, by Remark \ref{rem:height}\ref{rem:height(1)}, so are also the submonoids of $H$.

To start with, it is clear from Eq.~\eqref{equ:r.dim-monotone} that $\alpha + h \leq N$ when $a$ is an $\rfix_R$-unit of $H$ and hence $h = \allowbreak 0$. Therefore, we assume from now on that $a$ is an $\rfix_R$-non-unit. It follows that either $h \in \mathbb{N}^+$ or $h = \infty$. Accordingly, pick an integer $n$ in the (non-empty) interval $\llb 1, h \rrb$. By Remark \ref{rem:height}\ref{rem:height(2)} and Definition~\ref{def:r.fix-preorder} (applied to the $\rfix_R$-preorder on $H$), there then exist $x_1, \ldots, x_n \in R$ with $x_1 = a$ and $\rfix(x_k) \subsetneq \allowbreak \rfix(x_{k+1}) \subsetneq \allowbreak R$ for every $k \in \llb 1, n-1 \rrb$ (here again we use that the only $\rfix_R$-unit of $H$ is $1_R$). So, letting $x_{n+1} := 1_R$, we get from Remark \ref{rem:right-rickart-rings}\ref{rem:right-rickart-rings(2)} that there are non-zero right ideals $\mathfrak i_1, \ldots, \mathfrak i_n$ of $R$ such that 
\[
\rfix(x_{k+1}) = \rfix(x_k) \oplus \mathfrak i_k,
\qquad \text{for each }k \in \llb 1, n \rrb.
\]
Consequently, we see (by induction) that $
R = \rfix(x_k) \oplus \mathfrak{i}_{k} \oplus \cdots \oplus \mathfrak{i}_n$ for all $k \in \llb 1, n \rrb$; in particular, 
\begin{equation}\label{equ:2021-06-29-17:29}
R = \rfix(x_1) \oplus \mathfrak i_1 \oplus \cdots \oplus \mathfrak i_n = \rfix(a) \oplus \mathfrak i_1 \oplus \cdots \oplus \mathfrak i_n.
\end{equation}
Since $\mathfrak{i}_1, \ldots, \mathfrak i_n$ are non-zero right ideals of $R$, we therefore conclude from Eqs.~\eqref{equ:2021-04-28-23:55} and \eqref{equ:2021-06-29-17:29} that
\[
N = \rdim_R(R) = \rdim_R(\rfix(a)) + \sum_{k=1}^n \rdim_R(\mathfrak i_k) \ge \rdim_R(\rfix(a)) + n = \alpha + n,
\]
which shows in turn that $\alpha + h \leq N$ (as wished), because $n$ was an arbitrary integer in $\llb 1, h \rrb$ (note, in par\-tic\-u\-lar, that $h$ must be finite).
\end{proof}

\subsection{Von Neumann regular rings}
\label{sec:VNR-case}
By Theorem \ref{thm:rFix-quarks-in-right-rickart} and Remark \ref{rem:left-sing-implies-sing}, the $\rfix$-quarks of the monoid of zero divisors of large classes of right Rickart rings are idempotent. Given a Rickart ring $R$, it is then natural to ask whether also the $\rfix_R$-irreducibles of $R^\#$ of some degree $s \ge 2$ are idempotent. In the present subsection, we address this question in the class of von Neumann regular rings and show that the answer is positive under certain conditions (Theorem \ref{thm:5.8}). We start with some general considerations.

\begin{remarks}\label{rem:5.6-VNR}
\begin{enumerate*}[label=\textup{(\arabic{*})}]
\item\label{rem:5.6-VNR(1)} Let $e$ and $f$ be idempotents of a ring $R$ for which there exist right $R$-linear maps $\alpha \colon eR \to fR$ and $\beta \colon fR \to eR$ with the property that $\beta(\alpha(e)) = e$, and set $z := \alpha(e)$ and $y := \beta(f)$. We claim that $z = fze$, $y = eyf$, and $yz = e$: Incidentally, this results in a generalization of Proposition 4 in \cite[Ch.~III, Sect.~7]{Ja64}, where $eR$ and $fR$ are rather assumed to be isomorphic (as right $R$-modules). \\[0.05cm]

\indent{}For the claim, it is clear that $z \in fR$ and $y \in eR$, which gives in turn that $z = fz$ and $y = ey$ (by the as\-sump\-tion that $e^2 = e$ and $f^2 = f$). Moreover, we have from the right $R$-linearity of $\alpha$ and $\beta$ that
\end{enumerate*} 
\[
ze = \alpha(e) e = \alpha(e^2) = \alpha(e) = z
\quad\text{and}\quad
yf = \beta(f) f = \beta(f^2) = \beta(f) = y.
\]
It follows that $z = fz = ze = fze$ and $y = ey = yf = eyf$; and since $\beta(\alpha(e)) = e$ (by hypothesis), it is then immediate to check that $yz = \beta(f) z = \beta(fz) = \beta(z) = \beta(\alpha(e)) = e$.
\vskip 0.05cm

\begin{enumerate*}[label=\textup{(\arabic{*})},resume]
\item\label{rem:5.6-VNR(2)} 
Let $R$ be a von Neumann regular ring, and let $a \in R$ be such that $\rann(a) \cap aR = \allowbreak \{0_R\}$. 
Since, by Remark \ref{rem:right-rickart-rings}\ref{rem:right-rickart-rings(1)}, $R$ is a right Rickart ring, $\rann(a)$ and $\rfix(a)$ are both direct summands of $R_R$, viz., $\rann(a) = eR$ and $\rfix(a) = fR$ for some idempotents $e, f \in R$. Then $\mathfrak i := \rann(a) + aR = \allowbreak eR \oplus \allowbreak aR$; in particular, $\mathfrak i$ is a finitely generated right ideal of $R$, which, by \cite[Theorem 1.1]{Go79}, implies that 
$R = \allowbreak \mathfrak i \oplus \allowbreak \mathfrak j = \rann(a) \oplus \allowbreak aR \oplus \allowbreak \mathfrak j$ for a certain right ideal $\mathfrak j \subseteq R$.
Therefore, $fR = \rfix(a) \subseteq aR = gR$ for some idempotent $g \in R$, and by Proposition \ref{prop:2.2}\ref{prop:2.2(iv)}, this proves that $aR = \allowbreak \rfix(a) \oplus g_0R$ for some other idempotent $g_0 \in R$. So, putting it all together, we see that there is a Peirce basis $(e_1, e_2, e_3)$ of $R$ with $\rann(a) = e_1R$, $\rfix(a) = e_2R$, and $aR \subseteq e_2R \oplus e_3R = (1_R-e_1) R$.
\end{enumerate*}
\end{remarks}

Following K.\,R.~Goodearl and D.~Handelman (see the unnumbered definition on p.~80 and the notes on p.~94 of \cite{Go79}), we say that a von Neumann regular $R$ satisfies the \evid{comparability axiom} if, for all $u, v \in R$, there is an injective homomorphism of (right) $R$-modules either from $uR$ to $vR$ or from $vR$ to $uR$.

\begin{theorem}\label{thm:5.8}
Let $a$ be an $\rfix_R$-irreducible of the monoid of zero divisors of a Dedekind-finite, von Neumann regular ring $R$ satisfying the comparability axiom. Then $a$ is a coprimitive i\-dem\-po\-tent.
\end{theorem}

\begin{proof}
Since $R$ is Dedekind-finite and von Neumann regular, we are guaranteed by Remark \ref{rem:left-sing-implies-sing}\ref{rem:left-sing-implies-sing(2)} that every left or right zero divisor of $R$ is a zero divisor. By Proposition \ref{prop:4.6}\ref{prop:4.6(i)} and Remark \ref{rem:right-rickart-rings}\ref{rem:right-rickart-rings(1)}, it follows that $\rann(a)$ is an indecomposable $R$-module with $\rann(a) \cap \allowbreak aR = \{0_R\}$, and it is left to see that $a = a^2$. 
To start with, we have from Remark \ref{rem:5.6-VNR}\ref{rem:5.6-VNR(2)} that there is a Peirce basis $\bar{E} = (\bar{e}_1, \bar{e}_2, \bar{e}_3)$ of $R$ with
\begin{equation}
\label{equ:VNR-decomposition}
\rann(a) = \bar{e}_1 R \ne \{0_R\}, 
\quad
\rfix(a) = \bar{e}_2 R,
\quad\text{and}\quad
aR \subseteq (1_R-\bar{e}_1)R = \rfix(a) \oplus \bar{e}_3R.
\end{equation}
Suppose by way of contradiction that $a \ne a^2$. By Proposition \ref{prop:caratterizzazione-idemps} and Eq.~\eqref{equ:VNR-decomposition}, this means that $\bar{e}_3 \ne 0_R$.
We claim that there exists a Peirce basis $E = (e_1, e_2, e_3)$ of $R$ with $e_1 = \bar{e}_1$, $\rfix(a) \subseteq e_2R$, and $e_3 \ne 0_R$ and a pair of right $R$-linear maps $\alpha \colon e_3R \to e_1R$ and $\beta \colon e_1 R \to e_3 R$ such that $\beta(\alpha(e_3)) = e_3$. 

For, since $R$ satisfies the comparability axiom, there is an injective homomorphism of $R$-modules $\varphi$ either from $\bar{e}_3R$ to $\bar{e}_1R$ or from $\bar{e}_1R$ to $\bar{e}_3R$. In fact, we may assume that $\varphi$ is actually a map $\bar{e}_1R \to \bar{e}_3R$, or else we are done by taking $E := \bar{E}$, $\alpha := \varphi$, and $\beta := \varphi^{-1}$, where $\varphi^{-1}$ is the inverse of the co\-restriction of $\varphi$ to its own image. 
Consequently, we have from \cite[Theorem 1.1]{Go79} and Proposition \ref{prop:2.2}\ref{prop:2.2(iv)} that $
\{0_R\} \ne \allowbreak \varphi(\bar{e}_1) R = \allowbreak \varphi(\bar{e}_1 R) = \bar{e}_{3,2} R \subseteq \bar{e}_{3,1} R \oplus \bar{e}_{3,2} R = \bar{e}_3 R$
for some orthogonal idempotents $\bar{e}_{3,1}, \bar{e}_{3,2} \in R$, which finishes the proof of the claim upon taking $E := (\bar{e}_1, \bar{e}_2 + \bar{e}_{3,1}, \bar{e}_{3,2})$, $\alpha := \varphi^{-1}$, and $\beta := \varphi$. 

Now, it follows from the above and Remark \ref{rem:5.6-VNR}\ref{rem:5.6-VNR(1)} that there are $y \in e_3 R e_1$ and $z \in e_1 R e_3$ such that $yz = e_3$. Accordingly, define $b := e_3 + ae_2 + (a - 1_R)y$ and $c := 1_R - e_1 + z$. 
Since $E$ is a Peirce basis of $R$ and, by construction, $y \in Re_1$ and $z \in e_1 R e_3 \subseteq e_1 R (1_R-e_1)$, we get from Prop\-o\-si\-tion \ref{prop:2.2}\ref{prop:2.2(v)} that $c = c^2$ and hence $c \in R^\#$. Moreover, $e_3z = \allowbreak ae_2z = \allowbreak y(1_R-e_1) = 0_R$.
So, using that $yz = e_3$,
we find
%and $e_1 + e_2 + e_3 = 1_R$, we  find
\[
\begin{split}
bc 
& = e_3(1_R - e_1) + ae_2(1_R - e_1) + (a - 1_R)y(1_R - e_1) + e_3z + ae_2z + (a - 1_R)yz \\
& = e_3 + ae_2 + (a - 1_R)e_3 = 0_R + a(e_2 + e_3) = a(e_1 + e_2 + e_3) = a.
\end{split}
\]
We claim that $b$ and $c$ are both zero divisors of $R$, and each is strictly smaller than $a$ with respect to the $\rfix_R$-preorder on $R^\#$. This will contradict that $a$ is an $\rfix_R$-irreducible of $R^\#$ and complete the proof of the theorem. To begin, let $x \in \rfix(a)$. Since $\rfix(a) \subseteq e_2R$ (by construction), we have $x = e_2u$ for some $u \in R$. Moreover, $ye_2 = ze_2 = 0_R$ (recall that $y \in R e_1$ and $z \in Re_3$). Therefore,
\begin{equation}\label{equ:2021-07-17-16:43}
bx = e_3 e_2 u + a\,e_2^2u + (a-1_R)ye_2u = a\,e_2u = ax = x
\quad\text{and}\quad
cx = x - e_1 e_2 u + ze_2u = x.
\end{equation}
In a similar fashion, it is found that $ye_3 = ze_1 = 0_R$, which implies
\begin{equation}\label{equ:2021-07-17-16:44}
be_3 = e_3^2 + a\,e_2e_3 + (a-1_R)y\, e_3 = e_3
\quad\text{and}\quad
c\,e_1 = (1_R - e_1)e_1 + ze_1 = 0_R. 
\end{equation}
In particular, we get from Eqs.~\eqref{equ:2021-07-17-16:43} and \eqref{equ:2021-07-17-16:44} that 
\begin{equation}\label{equ:2021-07-17-17:25}
\rfix(a) \cup e_3 R \subseteq \rfix(b),
\quad
\rfix(a) \subseteq \rfix(c),
\quad\text{and}\quad
\rann(a) = e_1 R \subseteq \rann(c).
\end{equation}
In fact, it is not difficult to verify that $\rann(a) = \rann(c)$, which will come in handy in Remark \ref{rem:artinian-vnr}\ref{rem:artinian-vnr(1)}. For, let $v \in \rann(c)$. Since $v = e_1 v + e_2 v + e_3 v$ and $e_2 z = e_3 z = 0_R$ (recall that $z \in e_1 R$), we see that $0_R = (1_R-e_1)v + zv = e_2v + e_3v + zv$. But this is only possible if $0_R = e_2^2 v + e_2 e_3 v + e_2 z v = e_2 v$ and, likewise, $0_R = e_3 v$, so that $v = e_1v \in e_1 R$ and hence $\rann(c) \subseteq e_1 R = \rann(a)$.

Also, $\rfix(a) \cap e_3 R = \{0_R\}$, because $e_2 r = e_3 r'$ for some $r, r' \in R$ only if $e_2^2 r = e_2 e_3 r' = 0_R$ (recall that $\rfix(a) \subseteq e_2R$). Thus, we conclude from Eq.~\eqref{equ:2021-07-17-17:25}, similarly as in the proof of Lemma \ref{lem:2021-04-29-12:26}\ref{lem:2021-04-03-12:06(iii)}, that $
\rfix(a) \subsetneq \rfix(a) \oplus e_3R \subseteq \rfix(b)$,
where in the first inclusion we use that $e_3 \ne 0_R$. It only remains to check that $b \in R^\# \setminus \{1_R\}$ and $\rfix(a) \subsetneq \rfix(c)$.

For, we have from Eq.~\eqref{equ:VNR-decomposition} that $aR \subseteq (1_R-e_1)R$ (recall that $e_1 = \bar{e}_1$) and hence $e_1 a = 0_R$. It follows that $e_1 b = e_1e_3 + \allowbreak e_1ae_2 + e_1(a - 1_R)y = 0_R$ (recall that $y \in e_3Re_1$) and hence $0_R \ne e_1 \in \lann(b)$. As a result, $b$ is in $R^\# \setminus \{1_R\}$, for we have already noted that every left zero divisor of $R$ is a zero divisor.

Finally, assume $\rfix(a) = \rfix(c)$. Since $c = c^2$ (as proved earlier), we get from Proposition \ref{prop:caratterizzazione-idemps} %Eq.~\eqref{equ:ann&fix-of-idempotent} 
that $(1_R-e_1 + z)R = \rfix(a) \subseteq e_2 R$. In particular, this yields $1_R - e_1 + z = e_2 w$ for some $w \in R$ and hence $0_R = \allowbreak e_1 e_2 w = e_1 (1_R - e_1) + e_1 z = z$ (because $z \in e_1 R$). Thus, $0_R \ne e_3 = yz = 0_R$ (which is absurd).
\end{proof}

\begin{remarks}\label{rem:artinian-vnr}
Let $R$ be a Dedekind-finite, von Neumann regular ring satisfying the comparability axiom.

\vskip 0.05cm
\indent{}\begin{enumerate*}[label=\textup{(\arabic{*})}]
\item\label{rem:artinian-vnr(1)} 
Given a zero divisor $a \in R$ such that $\rann(a)$ is an indecomposable $R$-module and $a$ is not i\-dem\-po\-tent,
we gather from the proof of Theorem \ref{thm:5.8} that $a$ factors as the product $bc$ of two zero divisors $b, c \in R$ with $c = c^2$, $\rfix(a) \subsetneq \rfix(b)$, $\rfix(a) \subsetneq \rfix(c)$, and $\rann(a) = \allowbreak \rann(c)$; in particular, $\hgt(b) + \allowbreak \hgt(c) \leq \allowbreak \hgt(a)$, where $\hgt(\cdot)$ is the $\rfix_R$-height on $R^\#$ (cf.~Remark \ref{rem:5.3}). \\[0.05cm]

\indent{}\item\label{rem:artinian-vnr(2)} We have from Corollary  \ref{cor:2021-06-02-15:38} and Theorem \ref{thm:5.8} that, if $R^\#$ is an $\rfix_R$-artinian monoid, then each zero divisor $x \in R$ factors as a finite product of coprimitive idempotents of $R$. But, since we are guar\-an\-teed by Remark \ref{rem:left-sing-implies-sing}\ref{rem:left-sing-implies-sing(2)} that every non-unit of $R$ is a zero divisor, the artinianity of the $\rfix_R$-preorder on $R^\#$ is equivalent to the ascending chain condition on the principal right ideals of $R$, which, by \cite[Corollary 2.16]{Go79}, is in turn equivalent to $R$ being \evid{semisimple}, meaning that every right ideal of $R$ is a direct summand of $R_R$. This narrows down the class of von Neumann regular rings $R$ to which we can apply Corollary \ref{cor:2021-06-02-15:38} and leaves open the question whether a refinement of the techniques developed in the present work can handle a greater variety of cases. An affirmative answer is perhaps not so un\-reasonable, especially when considering that, by Example 8.1 in \cite{Go79}, there exist Dedekind-finite, von Neumann regular rings satisfying the comparability axiom that are not artinian and hence not semi\-simple (incidentally, the same example also shows that the characterization provided by Theorem \ref{thm:5.8} applies to a class of von Neumann regular rings that is \emph{strictly larger} than the class of semisimple rings).
\end{enumerate*}
\end{remarks}

Motivated by Remark \ref{rem:artinian-vnr}\ref{rem:artinian-vnr(2)}, we conclude this subsection by restricting our focus on idempotent fac\-tor\-i\-za\-tions in semisimple rings. We recall that semisimple rings are Dedekind-finite (see \cite[Exercise 3.10]{Lam03}), von Neumann regular (as is obvious from the definitions and \cite[Theorem 1.1]{Go79}), and noetherian (see \cite[Corollaries (2.6) and (3.7)]{Lam01}). Moreover, the ring of $n$-by-$n$ matrices over a skew field is semisimple (see, e.g., \cite[Theorem 3.3 and Corollary (3.7)]{Lam01}) and satisfies the comparability axiom (see \cite[Theorem 9.12 and Corollary 9.16]{Go79}). We will use these basic facts without further mention in what follows.

\begin{proposition}\label{prop:5.10}
Set $R := \mathcal M_n(K)$, where $K$ is a skew field. Every singular matrix $A \in R$ factors as a product of at most $\hgt(A)$ idempotent matrices of rank $n-1$, where $\hgt(\cdot)$ is the $\rfix_R$-height of $R^\#$.
\end{proposition}
\begin{proof}
It follows from Theorem \ref{thm:rFix-quarks-in-right-rickart} and Remarks \ref{rem:5.3} and \ref{rem:artinian-vnr}\ref{rem:artinian-vnr(1)} that a singular matrix $A \in R$ either is a coprimitive idempotent of $R$ or factors as the product $BC$ of two singular matrices $B, C \in R$ such that $C$ is an $\rfix_R$-quark of $R^\#$ and $\hgt(B) + \hgt(C) \leq \hgt(A)$.
On the other hand, we have from Proposition \ref{prop:rickart-strongly-fix-artinian} that $R^\#$ is a strongly $\rfix_R$-artinian monoid, for the right uniform dimension of a noetherian ring is finite (see, e.g., \cite[Corollary (6.7)(1)]{Lam99}).
We can therefore apply Theorem \ref{thm:3.5} with $s = 2$ and conclude that every $A \in R^\#$ factors as a product of at most $\hgt(A)$ coprimitive idempotents, which are in turn the idempotent matrices of $R$ of rank $n-1$ (Proposition \ref{prop:5.12}\ref{prop:5.12(i)}).
\end{proof}

The next corollary is now an immediate consequence of Propositions \ref{prop:rickart-strongly-fix-artinian} and \ref{prop:5.10}, when considering that, for a skew field $K$, the uniform dimension of $\mathcal M_n(K)$ as a right $\mathcal M_n(K)$-module equals to $n$ (see, e.g., \cite[Corollary 2.8 and Proposition 2.15]{ShCh07}).

\begin{corollary}
\label{cor:erdos-dawlings}
Every singular $n$-by-$n$ matrix with entries in a skew field is a product of at most $n$ idempotent matrices of rank $n-1$.
\end{corollary}

On the one hand, Corollary \ref{cor:erdos-dawlings} strengthens a result of Laffey \cite{Laf83} (in turn, a generalization of Erdos' classical theorem \cite{Er68}) on the existence of an idempotent factorization for singular matrices over a skew field; on the other, it provides a non-commutative extension of Dawlings' theorem \cite{Da81} on the existence of a factorization into idempotent factors of rank $n-1$ for the singular linear transformations of a vector space of finite dimension $n$ over a field.

The bound $n$ for the minimum length of a factorization into coprimitive idempotents of a singular $n$-by-$n$ matrix over a skew field is sharp, as it is known from the unnumbered corollary on the bottom of \cite[p.~81]{Ba78} that, for any integer $n \ge 2$ and any field $K$, there are $n$-by-$n$ singular matrices with entries in $K$ that cannot be written as a product of $n-1$ idempotents of $\mathcal M_n(K)$.

\begin{remark}\label{rem:semisimple-rings}
Let $R$ be a semisimple ring. By the Wedderburn-Artin theorem, there are a unique integer $k \ge 1$ and a unique tuple $(n_1, \ldots, n_k) \in \mathbb{N}^{\times k}$ with $1 \le n_1 \leq \cdots \leq n_k$ such that $R$ is i\-so\-mor\-phic (as a ring) to a direct product of the form $\prod_{i=1}^k \mathcal M_{n_i}(D_i)$, where the $D_i$'s are skew fields. 
Since ring homomorphisms map idempotents to i\-dem\-po\-tents, we can actually assume without loss of generality that $R = \prod_{i=1}^k R_i$, where for ease of notation $R_i := \mathcal M_{n_i}(D_i)$.
It then follows from Claim \ref{claimA} in Remark \ref{rem:idem-factorizations-in-direct-products} that $\mathcal E(R) = \prod_{i=1}^k \mathcal E(R_i)$. So we find $\mathcal E(R) = \prod_{i=1}^k R_i^\#$, because it is guaranteed by Corollary \ref{cor:erdos-dawlings} that $\mathcal E(R_i) = R_i^\#$ for each $i \in \llb 1, k \rrb$.
On the other hand, we gather from the comments under the proof of the aforementioned Claim \ref{claimA} that the minimum length $\ell(x)$ of an idempotent factorization of an element $x = \allowbreak (x_1, \ldots, x_k) \in \prod_{i=1}^k R_i^\#$ is bounded above by $\max_{1 \le i \le k} \ell_i(x)$, where $\ell_i(x)$ is the minimum length of an idempotent factorization of $x_i$ in $R_i^\#$. In particular, Corollary \ref{cor:erdos-dawlings} gives $\ell(x) \le \max(n_1, \ldots, n_k) = n_k$.
\end{remark}

As a complement to Remark \ref{rem:semisimple-rings}, note that, by \cite[Corollary 8.6]{Go79}, a semisimple ring $R$ satisfies the axiom of comparability only if $R = \mathcal M_n(K)$ for a skew field $K$, which explains why we cannot just apply Theorems \ref{thm:3.5} and \ref{thm:5.8} and conclude in general that the idempotents generate the zero divisors of $R$. 

\subsection{Endomorphism rings}
\label{sec:5.2-endo-rings}
Let $M$ be a (right) module over a ring $D$. In this subsection, we apply some of the results of Sect.~\ref{sec:4} to the monoid of zero divisors of the endomorphism ring $\End(M)$. 

Given $f \in \End(M)$, we denote by $\ker(f) := \{{x} \in M \colon f({x}) = {0}_M\}$ the \evid{kernel} of $f$ (as usual) and set
\begin{equation}\label{equ:fix(f)}
\fix(f) := \ker (\text{id}_M - f) = \{{x} \in M \colon f({x}) = {x}\}.
\end{equation}
Since $M$ is canonically a left mod\-ule over $\End(M)$, with multiplication given by the map $\End(M) \times M \to M \colon (g, {x}) \mapsto g({x})$, we have $\ker(g) = \{{x} \in M \colon g{x} = 0_M\}$ for every $g \in \End(M)$.
As a result, 
we gather from (the right analogue of) \cite[Lemma C.1(1)]{Nic-You03} that
\begin{equation}\label{equ:ker(f-f^2)}
\ker(f - f^2) = \ker(f) \oplus \ker(\id_M - f) = \ker(f) \oplus \fix(f).
\end{equation}
In particular, note that, if $f$ is idempotent, then $\ker(f-f^2) = M$. This leads to the next proposition, whose proof is along the same lines of the proof of Proposition \ref{prop:caratterizzazione-idemps} (we leave the details to the reader).

\begin{proposition}\label{prop:idempotent-endo}
An endomorphism $f$ of $M$ is idempotent in $\End(M)$ if and only if $\ker(f) \oplus \allowbreak \fix(f) = \allowbreak M$, if and only if $\fix(f) = f(M)$.
\end{proposition}
Building on these premises, we collect some elementary results that are folklore in module theory but for which we have not been able to find any convenient reference.
\begin{proposition}\label{prop:5.12}
Let $R$ be the endomorphism ring of $M$ and let $f \in R$. Suppose that $M = \ker(f) \oplus N$ for some submodule $N \subseteq M$ and denote by $p$ the projection of $M$ on $\ker(f)$ along $N$. The following hold:
\begin{enumerate}[label=\textup{(\roman{*})}]
\item\label{prop:5.12(i)} $\rann(f) = pR$, $\ker(f) = p(M)$, and $R = pR \oplus (\id_M-p)R$.
\item\label{prop:5.12(ii)} $\rann(f) \subseteq R_R$ is an indecomposable module if and only if so is $\ker(f)$ as a submodule of $M$.
\end{enumerate}
\end{proposition}

\begin{proof}
\ref{prop:5.12(i)}: Every projection of $M$ is an i\-dem\-po\-tent of $R$, see, e.g., \cite[Corollary 5.8]{An-Fu92}. Moreover, it is clear that $\ker(f) = p(M)$. In view of Eq.~\eqref{equ:direct-sum-decomposition-idemp}, it is therefore enough to show that $\rann(f) = pR$. 

Since $\ker(f) = p(M)$, we have $fp({x}) = f(p({x})) = {0}_M$ for each ${x} \in M$, which implies $fp = 0_R$ and hence $pR \subseteq \rann(f)$. As for the opposite inclusion, let $g \in \rann(f)$. Then $fg = {0}_R$ and hence $g(M) \subseteq \ker(f) = p(M)$. Since $p({y}) = {y}$ for all ${y} \in \ker(f)$, it follows that $g=pg \in pR$ and hence $\rann(f) \subseteq pR$. 

\vskip 0.05cm

\ref{prop:5.12(ii)}: Assume first that $\rann(f)$ is an indecomposable submodule of $R_R$ and suppose for a contradiction that $\ker(f) = N_1 \oplus N_2$ for some non-zero submodules $N_1, N_2 \subseteq M$. It follows from our hypotheses that
$M = \ker(f) \oplus N = N_1 \oplus N_2 \oplus N$.
Let $p_1$ be the projection of $M$ on $N_1$ along $N_2 \oplus N$, and $p_2$ be the projections of $M$ on $N_2$ along $N_1 \oplus N$. It is straightforward that $p = p_1 + p_2$ and $p_1 p_2 = p_2 p_1 = 0_R$, i.e., $p_1$ and $p_2$ are orthogonal idempotents of $R$ (recall \cite[Corollary 5.8]{An-Fu92}). 
So, putting it all together, 
we conclude from item \ref{prop:5.12(i)} and Proposition \ref{prop:2.2}\ref{prop:2.2(ii)} that $\rann(f) = pR = p_1R \oplus p_2R$. But this is impossible, because $p_1R$ and $p_2R$ are \emph{non-zero} right ideals of $R$ (and $\rann(f)$ is assumed to be indecomposable).

Conversely, let $\ker(f)$ be an indecomposable submodule of $M$ and suppose for a contradiction that $\rann(f)=\mathfrak{i}_1\oplus\mathfrak{i}_2$ for some non-zero right ideals $\mathfrak{i}_1, \mathfrak{i}_2 \subseteq R$. By item \ref{prop:5.12(i)}, $pR$ is a direct summand of $R_R$ with $pR = \rann(f)$ and $p(M) = \ker(f)$. So, again by Proposition \ref{prop:2.2}\ref{prop:2.2(ii)}, there exist \emph{non-zero} orthogonal i\-dem\-po\-tents $e_1, e_2 \in R$ with $p = e_1 + e_2$ such that $\mathfrak i_1 = e_1 R$ and $\mathfrak i_2 = e_2R$. This is however absurd, for it yields $\ker(f) = p(M) = e_1(M) \oplus e_2(M)$ and hence contradicts that $\ker(f)$ is indecomposable.
\end{proof}

\begin{proposition}\label{prop:5.13}
Let $M$ be a Rickart module and $R$ be its endomorphism ring, and let $f, g \in R$. Then $f \preceq g$ if and only if $\fix(g) \subseteq \fix(f)$, where $\preceq$ is the $\rfix$ preorder on $R$.
\end{proposition}

\begin{proof}
Since $M$ is a Rickart module, it follows from Eq.~\eqref{equ:fix(f)} and Remark \ref{rem:right-rickart-rings}\ref{rem:right-rickart-rings(1)} that $M = \fix(f) \oplus N_f = \fix(g) \oplus N_g$ for some submodules $N_f, N_g \subseteq M$. Let $p$ be the projection of $M$ on $\fix(f)$ along $N_f$, and $q$ be the projection of $M$ on $\fix(g)$ along $N_g$. By Definition~\ref{def:r.fix-preorder} and Proposition \ref{prop:5.12}\ref{prop:5.12(i)}, $f \preceq g$ if and only if $\rfix(g) \subseteq \rfix(f)$, if and only if $qR \subseteq pR$. We claim that $qR\subseteq pR$ if and only if $\fix(g)\subseteq\fix(f)$.

Assume first that $qR\subseteq pR$, i.e., $q=p\alpha$ for some $\alpha \in R$. Then $pq=p\alpha=q$. It follows that $q(x) = \allowbreak p(q(x)) \in p(M)$ for each $x \in M$, whence $\fix(g) = q(M) \subseteq p(M)=\fix(f)$. If, on the other hand, $\fix(g) \subseteq \allowbreak \fix(f)$, then $q(M)\subseteq p(M)$, i.e., for every $x \in M$ there exists ${y} \in M$ such that $q({x}) = p({y})$. But then $pq({x}) = p(q({x})) = p(p({y})) = p({y}) = q({x})$, whence $q =p q$ and $qR\subseteq pR$.
\end{proof}
Assume from now on (unless noted otherwise) that $D$ is a commutative PID and $M$ is a free module of finite rank $n$ over $D$, and denote by $R$ the endomorphism ring $\End(M)$. If $D$ is not a field, then $R$ is not von Neumann regular, hence the study of idempotent factorizations in $R^\#$ cannot be reduced to the cases covered by Sect.~\ref{sec:VNR-case}. On the other hand, since every commutative \textup{PID} is a right semi-hereditary ring, we are guaranteed by Remark \ref{rem:right-rickart-rings}\ref{rem:right-rickart-rings(1)} that $M$ is a Rickart module and $R$ is a right Rickart ring. It follows by Eq.~\eqref{equ:fix(f)} that $\fix(f)$ is a direct summand of $M$. So, if $f, g \in R$ and $\fix(f) \subseteq \fix(g)$, we get from \cite[Exercise 5.4(2), p.~76]{An-Fu92} that $\fix(g) = \fix(f) \oplus N$ for a certain submodule $N \subseteq M$. From here and Proposition \ref{prop:5.13}, arguing as in Proposition \ref{prop:rickart-strongly-fix-artinian}, we find
\begin{equation}\label{equ:rank-nullity-ineq}
\hgt(f)\leq {n} - \rk(\fix(f)), 
\end{equation}
where $\hgt(\cdot)$ is the $\rfix_R$-height of $R$ and $\rk(\cdot)$ denotes the \evid{rank} of a free module over $D$. Therefore, we gather from Remark \ref{rem:height}\ref{rem:height(1)} that every (multiplicative) submonoid of $R$ is strongly $\rfix_R$-artinian.

In addition, since $M$ is a noetherian module, we have from Remark \ref{rem:left-sing-implies-sing}\ref{rem:left-sing-implies-sing(3)} that $f\in R$ is a zero divisor if and only $f$ is a non-injective endomorphism of $M$, i.e., $\rk(\ker(f))\ge 1$. We thus conclude from the rank-nullity theorem that 
$R^\#=\{f\in H \colon \rk(f(M))\le n-1\}\cup\{\id_M\}$.

\begin{remark}\label{rem:coprimitive_endo}
Given $f\in R$, $\ker(f)$ is an indecomposable free $D$-module if and only if $\rk(\ker(f))=1$. This shows, in light of Proposition \ref{prop:idempotent-endo}, that $f$ is a coprimitive idempotent if and only if $\fix(f)=f(M)$ and $\rk(\ker(f))=1$, if and only if $\rk(f(M))=\rk(\fix(f))=n-1$ (by the rank-nullity theorem).
\end{remark}

We can now recover the characterizations of $\rfix_R$-quarks and $\rfix_R$-irreducibles obtained at the beginning of Sect.~\ref{subsec:characterizations} for a general right Rickart ring. We rephrase the results in terms of $\End(M)$.
%; $\preceq$ denotes the $\rfix$-preorder on $R$.

\begin{proposition}\label{prop:quark-irreducibles}
Let $f$ be a non-injective endomorphism of $M$. The following hold:
\begin{enumerate}[label=\textup{(\roman{*})}]
    \item\label{prop:quark-irreducibles(i)} $f$ is an $\rfix_R$-quark of $R^\#$ if and only if $f$ is a coprimitive idempotent.
    \item\label{prop:quark-irreducibles(ii)} $f$ is an $\rfix_R$-irreducible of $R^\#$ only if $\rk(\ker(f))=1$, and if $\rk(\ker(f))>1$ then $f=g \,h$ for some non-injective endomorphisms $g, h \in R$ with $\fix(f) \subsetneq \fix(g)$ and $\fix(f) \subsetneq \fix(h)$ such that $h$ is a coprimitive idempotent.
\end{enumerate}
\end{proposition}
\begin{proof}
Since $R$ is a right Rickart ring in which every left zero divisor is a zero divisor, \ref{prop:quark-irreducibles(i)} follows immediately from Theorem \ref{thm:rFix-quarks-in-right-rickart}, and \ref{prop:quark-irreducibles(ii)} follows from Propositions \ref{prop:4.6}\ref{prop:4.6(i)} and \ref{prop:5.13} and Remark \ref{rem:5.3}.
\end{proof}

Note that one cannot expect that every $\rfix_R$-irreducible of a certain degree $s \ge 2$ of $R^\#$ is idempotent. In fact, it is well known from classical results in matrix theory \cite{Cohn66,Co-Za-Za18} that, for some choices of the PID $D$, there is at least one element $f \in R^\#$ that does not factor as a product of idempotents of $R$. 
In the remainder of this section, we will however see that $R^\#$ is idempotent-generated when $D$ is a discrete valuation domain (Theorem \ref{thm:DVD-main-thm}), by proving that, in such case, the degree-$3$ $\rfix_R$-irreducibles of $R^\#$ are co\-primitive idempotents (Proposition \ref{dge1}).
 Here we follow \cite[Ch.~9]{Ati-MacDo} and let a \textsf{discrete valuation domain} (DVD) be a commutative \textup{PID} with a unique maximal ideal generated by a non-zero prime element, where an element $p \in D$ is \textsf{prime} if $p$ is a non-unit and $ab \in pD$, for some $a, b \in D$, if and only if $a \in pD$ or $b \in pD$.

\begin{lemma}\label{lem:2021-04-23-13:03}
If $D$ is a \textup{DVD}, $N$ is a torsion-free $D$-module, and $X$ is a non-empty finite set of linearly de\-pen\-dent vectors of $N$, then there is an $x \in X$ that lies in the submodule of $N$ generated by $X \setminus \{x\}$.
\end{lemma}

\begin{proof}
Let $x_1, \ldots, x_m$ be the elements of $X$. By hypothesis, $x_1 a_1 + \cdots + x_m a_m = 0_N$ for some non-zero scalars $a_1, \ldots, a_m \in D$. So, it suffices to show that $x_j$ lies in the submodule of $N$ generated by $X \setminus \{x_j\}$ for a certain $j \in \llb 1, m \rrb$.
Since $D$ is a \textup{DVD}, we gather from \cite[Proposition 9.2]{Ati-MacDo} that there is a non-zero $p \in D$ such that every non-zero $x \in D$ is of the form $u p^s$ for some $s \in \mathbb{N}$ and $u \in D^\times$. Consequently, there exist $s_1, \ldots, s_m \in \mathbb{N}$ and $u_1, \ldots, u_m \in D^\times$ such that $x_i = u_i p^{s_i}$ for every $i \in \llb 1, m \rrb$. Accordingly, we define $s := \min(s_1, \ldots, s_m)$ and, for each $i \in \llb 1, m \rrb$, we set $b_i := u_i p^{s_i - s}$. It follows that
\[
0_N = x_1 u_1 p^{s_1} + \cdots + x_m u_m p^{s_m} = (x_1 b_1 + \cdots + x_m b_m) \, p^s,
\]
which is only possible if $x_1 b_1 + \cdots + x_m b_m = 0_N$ (because $N$ is a torsion-free module). On the other hand, we have $s_j = s$, and hence $b_j = u_j \in D^\times$, for some $j \in \llb 1, m \rrb$. So letting $I := \llb 1, m \rrb \setminus \{j\}$, we obtain that $x_j = -\sum_{i \in I} x_i b_i u_j^{-1}$, i.e., $x_j$ lies in the submodule of $N$ generated by $X \setminus \{x_j\}$.
\end{proof}

We already know from Theorem \ref{thm:rFix-quarks-in-right-rickart} and Remark \ref{rem:left-sing-implies-sing}\ref{rem:left-sing-implies-sing(3)} that every $\rfix_R$-quark of $R^\#$ is a coprimitive idempotent. Now, we aim at characterizing the degree-$3$ $\rfix_R$-irreducibles of $R^\#$ when $D$ is a DVD. 

Note that if $n=1$, then $R$ is ring-isomorphic to $D$, whence $R^\# =\{0_R,1_R\}$ and $0_R$ is a coprimitive idempotent (see also the comments under Eq.~\eqref{equ:2021-06-17-17:44}). This explains the formulation of the next result.
\begin{proposition}\label{dge1}
Let $D$ be a \textup{DVD} and $M$ be a free $D$-module of finite rank $n \ge 2$, and set $R := \End(M)$. The following hold for a non-injective endomorphism $\alpha \in R$:
\begin{enumerate}[label=\textup{(\roman{*})}]
\item\label{dge1(i)} If $\rk(\ker(\alpha)) = 1$ and $\fix(\alpha) = \{0_M\}$, then $\alpha = \beta \gamma$ for some non-identity $\beta, \gamma \in R^\#$ such that $\beta$ is a co\-primitive idempotent and $\fix(\gamma) \ne \{0_M\}$.
\item\label{dge1(ii)} If $\alpha$ is a degree-$3$ $\rfix_R$-irreducible of $R^\#$, then $\alpha$ is a coprimitive idempotent of $R$.
\end{enumerate}
\end{proposition}
\begin{proof}
Denote by $d$ the rank of $\fix(\alpha)$ (recall that any submodule of a free module over a commutative \textup{PID} is free). We have already observed that every degree-$3$ $\rfix_R$-irreducible of $R^\#$ is, in particular, an $\rfix_R$-irreducible. So, by Proposition \ref{prop:quark-irreducibles}\ref{prop:quark-irreducibles(ii)}, we may assume $\rk(\ker(\alpha)) = 1$ both in \ref{dge1(i)} and in \ref{dge1(ii)}.
%. 

Since $M$ is a Rickart module (as noted in the comments under the proof of Proposition \ref{prop:5.13}), we 
get from Eq.~\eqref{equ:ker(f-f^2)} that 
$\fix(\alpha) \oplus \ker(\alpha)$ is a direct summand of $M$, namely,
$M = \fix(\alpha) \oplus Q \oplus \ker(\alpha)$ for a certain submodule $Q$ of $M$. Using that $\rk(M) = n \ge 2$, there is hence a basis $(e_1,\ldots,e_n)$ of $M$ with
$\fix(\alpha) = \allowbreak \bigoplus_{i=1}^d e_iD$,
$Q = \bigoplus_{i=d+1}^{n-1} e_i D$,
and $\ker(\alpha) = e_n D$ (note that $0 \le d < n$, with $d = 0$ if and only if $\fix(\alpha) = \{0_M\}$). So, there are uniquely determined coefficients $a_{i,j} \in D$ ($1 \le i, j \le n$) such that 
\begin{equation}\label{equ:centrata}
\alpha(e_j)=\sum_{i=1}^n e_i a_{i,j},
\qquad  \text{for every } i \in \llb 1, n \rrb.
\end{equation}
Accordingly, we define, for each $i \in \llb d+1, n \rrb$, 
\begin{equation}\label{equ:vectors-vi}
v_i := \sum_{j=d+1}^{n-1} e_j a_{i,j} \in Q.
%e_{d+1} D \oplus \cdots \oplus e_{n-1} D.
\end{equation}
%(note that, if $d \ge n-1$, then $v_i = 0_M$).
Clearly, $v_{d+1}, \ldots, v_n$ form a non-empty set of linear dependent vectors of $Q$ (with $v_n = 0_M$ in the limit case when $d = n-1$). So, by Lemma \ref{lem:2021-04-23-13:03}, there exists an index $\ell \in \llb d+1, n \rrb$ such that 
$v_{\ell} = \allowbreak \sum_{i \in I} v_i b_i$,
where $I := \llb d+1, n \rrb \setminus \{\ell\}$ and $b_i \in D$. As a result, we conclude from Eq.~\eqref{equ:vectors-vi} that
\begin{equation}\label{auxil_2}
a_{\ell,j}=\sum_{i \in I} a_{i,j} b_i, \qquad \text{for every } j \in \llb d+1, n-1 \rrb.  
\end{equation}
With these premises in place, we can now focus on the proof of items \ref{dge1(i)} and \ref{dge1(ii)}.

\vskip 0.05cm

\ref{dge1(i)}: Suppose that $\fix(\alpha)=\{0_M\}$. It then follows from the above that $d = 0$, and we distinguish two cases depending on whether $\ell \ne n$ or $\ell=n$.

\vskip 0.05cm 

\textsc{Case 1:} $\ell \ne n$. Let $\beta$ and $\gamma$ be the right $R$-linear maps $M \to M$ uniquely defined by taking
\[
	\beta(e_j) := \left\{
	\begin{array}{ll}
		0_M & \text{if } j = \ell\\
	e_{\ell} b_j + e_j & \text{if } j \in I
	\end{array}
	\right.
	\quad\text{and}\quad
	\gamma(e_j) :=
	\left\{
	\begin{array}{ll}
	e_{\ell}+\sum_{i \in I} e_i a_{i,\ell} & \text{if } j = \ell \\
	\sum_{i\in I} e_i a_{i,j}  & \text{if } j \in I\setminus\{n\} \\
	0_M & \text{if } j = n
	\end{array}
	\right.\!.
	\]
It is fairly obvious that $\beta$ and $\gamma$ are both singular endomorphisms of $M$. On the other hand, we get from Eqs.~\eqref{equ:centrata} and \eqref{auxil_2} (specialized to $d = 0$) that, for every $j \in \llb 1, n-1 \rrb$,
\begin{equation*}
\begin{split}
\beta(\gamma(e_j)) 
    & \displaystyle = \beta\!\left(\sum_{i \in I} e_i a_{i,j} \right) = \sum_{i \in I} \beta(e_i) \, a_{i,j} = \sum_{i \in I} (e_\ell b_i + e_i)\, a_{i,j} \\
    & \displaystyle = e_\ell 
    \sum_{i \in I} b_i a_{i,j} + \sum_{i \in I} e_i a_{i,j} \! \stackrel{\scriptsize\eqref{auxil_2}}{=} \! e_\ell a_{\ell, j} + \sum_{i \in I} e_i a_{i,j} = \sum_{i=1}^n e_i a_{i,j} 
    \! \stackrel{\scriptsize\eqref{equ:centrata}}{=} \! \alpha(e_j).   
    \end{split}
\end{equation*}
Since $\beta(\gamma(e_n))=0_M=\alpha(e_n)$, it follows that $\alpha = \beta\gamma$. 
Moreover, $\beta$ is a co\-primitive idempotent of $R$ (by Remark \ref{rem:coprimitive_endo}), since $\{e_\ell b_i + e_i \colon i \in I\}$ is a basis of $\fix(\beta)$ and $\beta(e_\ell) = 0_M$; and $\fix(\gamma) \ne \allowbreak \{0_M\}$, since $(\id_M-\gamma)(M)\subseteq \bigoplus_{i\in I} e_iD$ and hence $\rk(\fix(\gamma)) \ge 1$ (by the rank-nullity theorem). 

\vskip 0.05cm

\textsc{Case 2:} $\ell=n$. Eq.~\eqref{auxil_2} becomes $a_{n,j}=\sum_{i=1}^{n-1} a_{i,j} b_i$ for every $j\in\llb 1,n-1\rrb$. By Proposition \ref{prop:conjugation}, we may assume that $\alpha(M) \subseteq \bigoplus_{i=1}^{n-1} e_i D$ (otherwise, let $\eta$ be the automorphism of $M$ uniquely determined by sending $e_i$ to $e_i + e_n b_i$ for $i \in \llb 1, n-1 \rrb$ and to $e_n$ for $i = n$, and replace $\alpha$ with $\eta^{-1} \alpha\, \eta$). Accordingly, we take $\beta$ and $\gamma$ to be the endomorphisms of $M$ uniquely defined by
\[
\beta(e_j) := \left\{
\begin{array}{ll}
e_j & \text{if } j \in \llb 1,n-1\rrb \\
\alpha(e_1)-e_1 & \text{if } j = n
\end{array}
\right.\quad\text{and}\quad
\gamma(e_j) :=
\left\{
\begin{array}{ll}
e_1 + e_n & \text{if } j = 1\\
\alpha(e_j) & \text{if } j \in \llb 2,n \rrb
\end{array}
\right..
\]
It is evident that $\beta(M) \subseteq \bigoplus_{i=1}^{n-1} e_i D$ and $\gamma(e_n) = 0_M$, so that $\beta$ and $\gamma$ are $\rfix$-non-units of $M$. Moreover, $\beta(e_j)=e_j$ for $j\in\llb 1,n-1\rrb$ (hence $\beta$ is a coprimitive idempotent) and $\gamma(e_1+e_n)=\gamma(e_1)+\gamma(e_n)= \allowbreak e_1+e_n$.
Lastly, it is routine to check that $\alpha = \beta\gamma$. 

\vskip 0.05cm

\ref{dge1(ii)}: Let $\alpha$ be a degree-$3$ $\rfix_R$-irreducible of $R^\#$ and assume for the sake of a contradiction that $d := \rk(\fix(\alpha)) < n-1$ (implying that $\fix(\alpha)\subsetneq\im(\alpha)$). Again, we distinguish two cases.
%, depending on whether $\ell = n$.

\vskip 0.05cm
\textsc{Case 1:} $\ell\ne n$. Let $\beta$ and $\gamma$ be the endomorphisms of $M$ uniquely determined by
\[
\beta(e_j) := \left\{
\begin{array}{ll}
e_j & \text{if } j\in \llb 1,d \rrb \\
0_M & \text{if } j = \ell\\
e_j+ e_{\ell}b_j & \text{if } j\in I
\end{array}
\right.
\quad\text{and}\quad
\gamma(e_j) :=
\left\{
\begin{array}{ll}
e_j & \text{if } j\in \llb 1,d \rrb\cup\{n\} \\
\alpha(e_j)-e_\ell a_{\ell ,j} & \text{if } j\in \llb d+1,n-1 \rrb
\end{array}
\right.;
\]
and denote by $\delta$ the projection of $M$ on $\fix(\alpha) \oplus Q$ along $\ker(\alpha)$.
Each of $\beta$, $\gamma$, and $\delta$ is a non-injective map, because $e_\ell \in \ker(\beta)$, $e_n \in \ker (\delta)$, and $\gamma(M) \subseteq \bigoplus_{j \in J} e_j D$ with $J := \llb 1, n \rrb \setminus \{\ell\}$. Moreover, it is readily seen that $\fix(\alpha)\cup\{e_j + e_\ell b_j \colon j \in I\}$ is a basis of $\fix(\beta)$, $\fix(\alpha) \cup \{e_n\}$ is contained in $\fix(\gamma)$, and $\fix(\alpha) \cup \{e_{d+1},\ldots, \allowbreak e_{n-1}\}$ is a basis of $\fix(\delta)$. It follows that $\beta$, $\gamma$, and $\delta$ are non-injective endomorphisms of $M$ each of which is (strictly) smaller than $\alpha$ with respect to the restriction to $R^\#$ of the $\rfix_R$-preorder on $R$. Moreover, arguing as in the proof of \ref{dge1(i)}, we gather from Eq.~\eqref{auxil_2} that  $\beta\gamma\delta(e_j)=\alpha(e_j)$ for every $j \in \llb 1, n \rrb$. But this is in contradiction to the fact that $\alpha$ is a degree-$3$ $\rfix_R$-irreducible of $R^\#$. (Incidentally, note that $\beta$ and $\delta$ are both coprimitive idempotents of $R$.)

\vskip 0.05cm

\textsc{Case 2:} $\ell=n$. Similarly as in Case 2 from the proof of item \ref{dge1(i)}, we may assume $\alpha(M) \subseteq \allowbreak \bigoplus_{i=1}^{n-1} e_i D$. Accordingly, we take $\beta$ and $\gamma$ to be the endomorphisms of $M$ uniquely defined by
\[
\beta(e_j) := \left\{
\begin{array}{ll}
e_j & \text{if } j\in\llb 1,n-1\rrb \\
\alpha(e_{d+1})-e_{d+1} & \text{if } j = n
\end{array}
\right.
\quad\text{and}\quad
\gamma(e_j) :=
\left\{
\begin{array}{ll}
\alpha(e_j) & \text{if } j \in \llb 1,n \rrb \setminus \{d+1\} \\
e_{d+1}+e_n & \text{if } j = d+1
\end{array}
\right..
\]
The desired conclusion follows arguing as in the last part of the proof of item \ref{dge1(i)}.
\end{proof}

A closer look at Propositions \ref{prop:quark-irreducibles} and \ref{dge1} leads to the next result.

\begin{theorem}\label{thm:DVD-main-thm}
Every non-injective endomorphism $\alpha$ of a free module $M$ of finite rank $n\ge 2$ over a commutative DVD is a product of $2n-2$ or fewer coprimitive idempotents of $\End(M)$.
\end{theorem}

\begin{proof}
Let $\hgt(\cdot)$ be the $\rfix_R$-height of $R^\#$ and $\ell(f)$ be the minimum number of factors in a factorization of a non-injective endomorphism $f$ of $M$ into coprimitive idempotents, where $R := \End(M)$. 
We have from Propositions \ref{prop:quark-irreducibles} and \ref{dge1}\ref{dge1(i)} and the proof of Proposition \ref{dge1}\ref{dge1(ii)} that either $f$ is an $\rfix_R$-quark of $R^\#$ or it factors as a product of $2$ or $3$ $\rfix_R$-non-units of $R^\#$ the sum of whose $\rfix_R$-heights is no larger than $\hgt(f)$ or $\hgt(f)+1$, resp.
It follows by Theorem \ref{thm:3.5} (applied with $s = 3$) that $\ell(f) \le 2\hgt(f) - 1$. 

Now, if $\rk(\fix(\alpha)) \ge 1$, we get from Eq.~\eqref{equ:rank-nullity-ineq} 
that $\ell(\alpha) \le 2(n-1)-1=2n-3$. If, on the other hand, $\rk(\fix(\alpha)) = 0$, then we see from the proof of Proposition \ref{dge1} that $\alpha = \beta\gamma$, where $\beta$ is a coprimitive idempotent and $\gamma$ is an $\rfix_R$-non-unit of $R^\#$ with $\hgt(\gamma) \le n-1$; therefore, we conclude from the previous case that $\ell(\alpha) \le 1 + \ell(\gamma) \le 1+2n-3 = 2n-2$.
\end{proof}
\begin{corollary}\label{DVD_matrix}
Every singular $n$-by-$n$ matrix over a commutative \textup{DVD}, with $n\ge 2$, is a product of $2n-2$ or fewer idempotent matrices of rank $n-1$.
\end{corollary}

\begin{proof}
This is simply a reformulation of Theorem \ref{thm:DVD-main-thm} in the language of matrices.
\end{proof}

Corollary \ref{DVD_matrix} incorporates a classical result of Fountain \cite[Theorem 4.1]{Fo91} on the \emph{existence} of a factorization into idempotent matrices of rank $n-1$ for every $n$-by-$n$ singular matrix over a commutative \textup{DVD}. But the upper bound on the minimum length of such a factorization offered by the same corollary is apparently new, and it is sharp for $n=2$ (it suffices to note that, for $n \ge 2$, the ring of $n$-by-$n$ matrices over an arbitrary ring contains at least one singular non-i\-dem\-po\-tent matrix, namely, the strictly upper triangular matrix whose non-diagonal entries are all equal to the identity).

\section{Closing remarks}\label{sec:6}

Many questions remain open. On the one hand, it would be interesting to study various arithmetic invariants that can be naturally attached to the ``minimal idempotent factorizations'' of the elements of an idempotent-generated monoid along the lines of what is done in \cite[Sect.~4]{An-Tr18} for atomic factorizations in an arbitrary monoid. On the other hand, it might be rewarding to try and extend the techniques developed in Sects.~\ref{sec:3} and \ref{sec:4} so as to prove that the monoid of zero divisors of a broader class of right Rickart rings than covered in Sect.~\ref{subsec:characterizations} is idempotent-generated.

For instance, let $R = \mathcal M_2(\mathbb{Z})$ be the ring of $2$-by-$2$ matrices over the integers. It is well known that every singular matrix over a commutative euclidean domain is a product of idempotent matrices \cite[Theorem 2]{Laf83}. In particular, it was proved by Fountain in \cite[Theorem 4.6]{Fo91} that every singular $A \in R$ is a product of non-zero idempotent matrices. However, the $\rfix_R$-preorder on the monoid $R^\#$ of zero divisors of $R$ cannot be the right preorder to consider if one's goal is to recover Fountain's result as a corollary of Theorem \ref{thm:3.4}.
In fact, if for some integer $s \ge 2$ the degree-$s$ $\rfix_R$-irreducibles of $R^\#$ were idempotent, then we would get from  Eq.~\eqref{equ:rank-nullity-ineq} that every singular matrix in $R$ factors as a product of at most $s$ idempotents. But this is impossible, since it was proved by Laffey \cite[Sect.~1]{Laf89} that, for any integer $N \ge 1$, there exists a singular matrix $A \in R$ which does not factor as the product of fewer than $N$ idempotents. (Surprisingly enough, this does not happen in $\mathcal M_n(\mathbb{Z})$ for $n \ge 3$, see \cite[Sect.~2]{Laf89}.)

\section*{Acknowledgements}

The first author's work was supported by the ``Ernst Mach Grant --- Worldwide'' of Austria's Agency for Education and Internationalisation (OeAD) under the project No.~ICM-2020-00054, and by the European Union's Horizon 2020 research and innovation programme under the Marie Sk\l{}odowska-Curie grant agreement No.~101021791.

The authors are grateful to Alfred Geroldinger and Daniel Smertnig (University of Graz, AT) for many precious comments, to Pace Nielsen (Brigham Young U\-ni\-ver\-si\-ty, US) for the proof of Lemma \ref{lem:2021-05-05-11:04} (see \href{https://mathoverflow.net/a/394965/16537}{\texttt{https://mathoverflow.net/a/394965/16537}}), and to an anonymous referee for careful reading and very useful editorial comments.

\end{document}